\documentclass[reqno, 11pt]{amsart}
\usepackage[dvipsnames,prologue]{xcolor} 
\usepackage{calc,graphicx,amsfonts,amsthm,amscd,amsmath,amssymb,enumerate,dsfont,mathrsfs,paralist,bbm,tikz}
\usetikzlibrary{patterns,arrows}
\usepackage{subfig}
\usepackage{pdfsync}
\usepackage{stmaryrd}
\usepackage{amsrefs}
\usepackage{etoolbox}
\usepackage{marginnote}
\usepackage{hyperref}
\AtBeginEnvironment{biblist}{\catcode`\#=12 }

\usetikzlibrary{shapes.misc}
\usetikzlibrary{shapes.geometric}
\usetikzlibrary{arrows.meta}
\usetikzlibrary{snakes}

\tikzset{cross/.style={cross out, draw=black, minimum size=2*(#1-\pgflinewidth), inner sep=0pt, outer sep=0pt}, cross/.default={1pt}}

\usepackage{booktabs}
\newcommand{\ie}{\hbox{\it i.e.\ }}

\reversemarginpar
\newlength\fullwidth
\numberwithin{equation}{section}

\DeclareMathSymbol{\leqslant}{\mathalpha}{AMSa}{"36} 
\DeclareMathSymbol{\geqslant}{\mathalpha}{AMSa}{"3E} 
\DeclareMathSymbol{\eset}{\mathalpha}{AMSb}{"3F}     
\newcommand{\sumtwo}[2]{\sum_{\substack{#1 \\ #2}}} 

\def\1{\ifmmode {1\hskip -3pt \rm{I}} \else {\hbox {$1\hskip -3pt \rm{I}$}}\fi}

\newcommand{\var}{\operatorname{Var}}
\newcommand{\cov}{\operatorname{Cov}}

\newcommand{\tc}{\thinspace |\thinspace}

\newcommand{\trel}{T_{\rm rel}}

\newcommand{\grad}{\nabla}

\renewcommand{\l}{\lambda}
\renewcommand{\L}{\Lambda}

\renewcommand{\l}{\lambda}

\renewcommand{\d}{\delta}

\newcommand{\g}{\gamma}

\newtheorem{theorem}{Theorem}[section]
\newtheorem*{theorem*}{Theorem}
\newtheorem{lemma}[theorem]{Lemma}
\newtheorem{proposition}[theorem]{Proposition}
\newtheorem{corollary}[theorem]{Corollary}
\newtheorem{remark}[theorem]{Remark}

\newtheorem{claim}[theorem]{Claim}
\newtheorem{definition}[theorem]{Definition}
\newtheorem{maintheorem}{Theorem}

\newtheorem*{question*}{Question}

\newtheorem*{remark*}{Remark}
\newtheorem*{idefinition*}{Definition}

\newtheorem{observation}{Observation}

\newcommand{\cA}{\ensuremath{\mathcal A}}
\newcommand{\cB}{\ensuremath{\mathcal B}}
\newcommand{\cC}{\ensuremath{\mathcal C}}
\newcommand{\cD}{\ensuremath{\mathcal D}}

\newcommand{\cF}{\ensuremath{\mathcal F}}
\newcommand{\cG}{\ensuremath{\mathcal G}}

\newcommand{\cL}{\ensuremath{\mathcal L}}
\newcommand{\cM}{\ensuremath{\mathcal M}}

\newcommand{\bbL}{{\ensuremath{\mathbb L}} }

\newcommand{\bbN}{{\ensuremath{\mathbb N}} }

\newcommand{\bbR}{{\ensuremath{\mathbb R}} }

\newcommand{\bbZ}{{\ensuremath{\mathbb Z}} }

    \let\d=\delta  
 \let\g=\gamma       \let\l=\lambda
      \let\o=\omega      
  \let\s=\sigma    
\let\y=\upsilon \let\x=\xi 
     \let\L=\Lambda 
\let\O=\Omega      

\begin{document}
\title[]{Diffusive scaling of the Kob-Andersen model in $\bbZ^d$} 
\author[F. Martinelli]{F. Martinelli}
\email{martin@mat.uniroma3.it}
\address{Dipartimento di Matematica e Fisica, Universit\`a Roma Tre, Largo S.L. Murialdo 00146, Roma, Italy}
\author[A. Shapira]{A. Shapira}
\email{assafshap@gmail.com}
\address{LPSM UMR 8001, Universit\'e Paris Diderot\protect\\CNRS, Sorbonne Paris Cit\'e\protect\\
75013 Paris, France}
\author[C. Toninelli]{C. Toninelli}
\email{toninelli@ceremade.dauphine.fr }
\address{CEREMADE UMR 7534, Universit\'e Paris-Dauphine\protect\\CNRS, PSL Research University\protect\\Place du Mar\'echal de Lattre de Tassigny, 75775 Paris Cedex 16, France}
\thanks{ This work has been supported by the ERC Starting Grant 680275 ``MALIG'', ANR-15-CE40-0020-01
and by PRIN 20155PAWZB ``Large Scale Random Structures''. 
 }

\begin{abstract}
We consider the Kob-Andersen model, a cooperative lattice gas with kinetic constraints which has been widely analysed in the physics literature in connection
with the study of the liquid/glass transition.  We consider the model in a finite box of linear size $L$ with sources at the boundary. Our result, which holds in any dimension and significantly improves upon previous ones, establishes for any positive vacancy density $q$ a purely diffusive scaling of the relaxation time $\trel(q,L) $ of the system.  Furthermore,  as $q\downarrow 0$ we prove upper and lower bounds on $L^{-2}\trel(q,L)$ which agree with the physicists belief that the dominant equilibration mechanism is a cooperative motion of rare large droplets of vacancies.
 The main tools
combine a recent set of ideas and techniques developed to
establish universality results
for kinetically constrained spin models, with methods from bootstrap percolation, oriented
percolation and canonical flows for Markov chains.
\end{abstract}

\maketitle

{\sl MSC 2010 subject classifications}: {60K35}
{60J27}

{\sl Keywords}: Kawasaki dynamics, spectral gap, kinetically constrained models

\section{Introduction}
 

Kinetically constrained lattice gases (KCLG) are interacting particle systems on the integer lattice $\mathbb Z^d$ with hard core
exclusion, i.e.\@
with the constraint that on each site there is at most one particle.
A configuration is therefore defined by giving for each site
$x\in\mathbb Z^d$ the occupation variable $\eta_x\in \{0,1\}$, which
represents an empty or occupied site respectively.
The evolution is
given by a continuous time Markov process of Kawasaki type, which allows
the exchange of the occupation variables across a bond $e=(x, y)$ of
neighbouring sites $x$ and $y$ with rate $c_{e}(\eta)$  (bonds are  non oriented, namely $(x,y)\equiv (y,x)$ and $c_{yx}(\o)=c_{xy}(\o)$). This exchange rate equals one
if the current configuration satisfies an a priori
specified local constraint and zero otherwise. In the former case we say that the exchange is \emph{legal}. A key feature of the constraint is that it does not depend 
on the occupation variables $\eta_x,\eta_y$
and therefore for any $q\in[0,1]$
detailed balance w.r.t.\@ $(1-q)$-Bernoulli product
measure $\mu$ is verified. Thus, $\mu$ is an
invariant reversible measure for the process. However, at variance with 
the simple symmetric exclusion process (SSEP), that corresponds to the case in which the constraint is always verified, KCLG have several
other invariant measures. This is related to the fact that due to the constraints 
there exist \emph{blocked configurations}, namely configurations for which some
exchange rates remain zero forever.

KCLG have been introduced in physics literature (see \cites{Ritort,GST} for a
review) to model the liquid/glass transition that occurs when a liquid is suddenly cooled.
In
particular they were devised to mimic the fact that the motion of a
molecule in a low temperature (dense) liquid can be inhibited by the geometrical
constraints created by the surrounding molecules.  Thus, to encode this local caging mechanism, the exchange rates  of KCLG require 
a minimal number of empty sites
in a certain neighbourhood of $e=(x,y)$ in order for the exchange at $e$
to be allowed.  There exists also a non-conservative version of KCLG,
the so called Kinetically Constrained Spin Models, which feature a
Glauber type dynamics and have been recently studied in several works
(see e.g. \cites{CMRT,MT,MMT} and references therein).

In this paper we focus on the class of KCLG which has been most studied in physics literature,  the so-called Kob-Andersen (KA) models \cite{KA}. 
Each KA model leaves on
$\mathbb Z^d$, with $d\geq 2$, and is characterised by an integer parameter $k$ with $k\in[2,d]$. The  nearest neighbour exchange rates  are defined as follows: $c_{xy}(\eta)=1$ 
 if at least $k-1$ neighbours of $x$ different from
$y$ are empty and at least $k-1$ neighbours of $y$ different from $x$
are empty too, $c_{xy}=0$ otherwise. 
The name  {\sl KA-$k$f model} is used in the literature to refer to the model with parameter $k$ \footnote{Here $f$ stands for "facilitation", since $k$ denotes the minimal number of empty sites to allow motion.}.  The choices $k=1$ and $k>d$ are discarded: $k=1$ would correspond to SSEP;    $k>d$ would yield the existence of finite clusters of particles which are blocked, and therefore for this choice at any $q<1$ the infinite volume process would  not be ergodic \footnote{Here ergodic means that  
 zero is  a simple eigenvalue for the
generator of the Markov process in $\bbL_2(\mu)$.} . 
For example for $k=3,d=2$ a  $2\times 2$ square fully occupied by particles is blocked: none of these particles can ever jump to their neighboring empty positions. \\
In \cite{TBF} it has been proven that for all $k\in[2,d]$ the infinite
volume KA-$k$f models are ergodic for all $q\in(0,1]$,  
thus disproving previous conjectures \cites{KA,FMP,KPS}
on the occurrence of an ergodicity breaking transition at $q_c>0$ based on numerical simulations.
In \cite{BT} it has been proved that for all $q\in(0,1]$ 
 the rescaled position of a marked particle at time $\epsilon^{-2} t$ converges as $\epsilon\to 0$, to a $d$-dimensional Brownian motion with non-degenerate diffusion matrix.
 This again disproves
 a conjecture that had been put forward in physics literature
 on the occurrence of a diffusive/non-diffusive transition at a finite critical density $q_c>0$
 \cites{KA,KPS}.
 Motivated by the fact that numerical simulations \cites{KA,MP} suggest the possibility of an anomalous slowing down at high density,
 in \cites{CMRT-KA} the relaxation time $T_{rel}$ (namely the inverse of the spectral spectral gap) has been studied. For KA-$2$f in dimension  $d=2$ it has been proved that in a box of linear size $L$ with
boundary sources, $T_{rel}$ is upper bounded by  $L^2\log L$ 
 at any $q\in(0,1]$. The same technique  
 can be extended to establish an analogous upper bound for all choices of $d$ and $k\in[2,d]$.  By using this result in \cite{CMRT-KA} it is also proved that  the infinite volume time auto-correlation of
local functions decays at least as $1/t$   modulo logarithmic corrections \cite{CMRT-KA}.
A lower bound as $1/t^{d/2}$ follows by comparison with SSEP.

The description of the state of the art for KCLG would not be complete without  mentioning that a purely diffusive scaling $L^2$ for the inverse of the spectral gap has  been established for some KCLG \cites{BertiniToninelli,GLT,nagahata}, with and without boundary sources. However, all the models considered in these papers belong to the so called class of {\sl non-cooperative} KCLG, namely models for which the constraints are such that it is possible to construct a finite group of vacancies, the mobile cluster, with
 two key properties. (i) For any configuration it is possible to move the mobile cluster to any other
position in the lattice by a sequence of allowed exchanges; (ii) any nearest neighbour exchange is allowed if
the mobile cluster is in a proper position in its vicinity.
The existence of  finite mobile clusters
is a key tool in the analysis of non-cooperative KCLG and allows the application of some  techniques (e.g. paths arguments)  developed for SSEP. 
  It is immediate to verify that instead, for all $k\in[2,d]$, KA models belong to the {\sl cooperative} class, which contain all models that are not non-cooperative.
For example for $k=d=2$, one can easily check that 
there cannot exist a  finite mobile cluster 
by noticing that any
a fully occupied double
column which spans the lattice can never be destroyed. \\
Besides being a challenging mathematical issue,  developing a new set of techniques to prove a purely diffusive scaling  for KA and for cooperative models in general is important  from the point of view of the modelization of the liquid/glass transition, since in this context cooperative models are undoubtedly
the most relevant ones. Indeed, very roughly speaking, non cooperative models behave like a rescaled SSEP with the mobile cluster playing the role of a single vacancy and are less suitable to describe the rich behavior of glassy dynamics. 

 Here we  significantly improve upon the existing results, by establishing  $\trel(L)\approx L^2$ for all KA-$k$f models  in a finite box of side $L$ of $\bbZ^d$, $d\ge 2$,
with sources at the boundary (Theorem \ref{maintheorem:1}). This is the first result establishing a pure diffusive
scaling 
 for a cooperative KCLG. 
The technique that we develop, which is  completely different from the one in \cite{CMRT-KA}, 
combines  a set of ideas and techniques recently developed by two of the authors to
establish universality results
for kinetically constrained spin models \cite{MT}, with methods from oriented
percolation and canonical flows for Markov chains.
 Although we have applied our technique for KA models, 
we expect our tools to be robust enough to be extended to analyse other KCLG in the ergodic regime.\\
Our main result (cf. Theorem \ref{maintheorem:1} ) establishes upper and lower bounds on $\trel$ of the form $C_-(q) \times L^2\le \trel(q,L)\le C_+(q)\times L^2$ with two parameters $C_+,C_-$ that diverge as $q\to 0$ . Remarkably, the divergence of both $C_-(q)$ and $C_+(q)$ has the same leading behavior, and it is qualitatively in agreement with that conjectured by the physicists \cite{TBF} and based on the assumption that the dominant mechanism driving the system to equilibrium is a complex cooperative motion of rare large droplets of vacancies.

The plan of the paper is the following. In Section \ref{sec:model} we introduce the notation and the results. In Section \ref{sec:core} we prove the upper bound on the relaxation time in several steps. We start by performing a coarse graining (Section \ref{sec:coarse-graining}), and proving a coarse-grained constrained Poincar\'e inequality (Section \ref{sec:long}). A key ingredient for this proof is the probability that a certain good event has a large probability, a result that is proved in Section \ref{sec:tools} by using tools from supercritical oriented percolation.
In Section \ref{sec:paths} and \ref{sec:longKA} we use canonical flows techniques in order to bound from above the r.h.s. of the coarse-grained Poincar\'e inequality with the Dirichlet form of KA model, and we conclude by using the variational characterization of the spectral gap. Finally, in Section \ref{sec:lowerbound} we prove the lower bound on the relaxation time, finding an appropriate test function and using again the variational characterization of the gap.

\section{The Kob-Andersen model and the main result}\label{sec:model}

Given an integer $L,$ and a parameter $q\in (0,1)$, we let $\Lambda=[L]^d$
$$\partial \Lambda=\{x\in \Lambda: \exists y\notin\Lambda \text{
  with } \|x-y\|_1=1\}.$$
and consider the probability space $(\Omega_\Lambda, \mu_\Lambda)$
where
\[
\Omega_\L = \left\{\eta \in \{0,1\}^{\bbZ^d}\,:\,\eta_x=0 \text{ for all }x\notin\L \right\}
\]
and $\mu_\Lambda$ is the product
Bernoulli(1-q) measure. Given $\eta\in \Omega_\Lambda$ and
$V\subset \Lambda,$ we shall say that $V$ is empty (for $\eta$) if
$\eta_x=0\ \forall x\in V.$

Fix an integer $k\in [2,d]$ and, for
any given a pair of nearest neighbour
sites $x,y$ in $\Lambda,$ write $c_{xy}(\cdot)$ for the indicator of
the event that both $x$ and $y$ have at least $k-1$ empty neighbours
among their nearest neighbours in $\Lambda$ without 
counting $x,y$ 
\begin{equation}\label{constraint}
c_{xy}(\eta)=
\begin{cases}
1 & \text{if }\sum_{z:\|x-z\|_1=1, z\neq y}(1-\eta_z)\geq k-1\\ &\text{ and }\sum_{z: \|y-z\|_1=1, z\neq x}(1-\eta_z)\geq k-1,\\
0&\text{otherwise}. 
\end{cases}
\end{equation}
and set
$$
\eta^{xy}_z: = \left\{
\begin{array} {ll}
\eta_z & \mbox{if } z \notin \{x,y\} \\
\eta_x & \mbox{if } z=y\\
\eta_y & \mbox{if } z=x.
\end{array}
\right.
$$

$$
\eta^{x}_z: = \left\{
\begin{array} {ll}
\eta_z & \mbox{if } z \neq x \\
1-\eta_x & \mbox{if } z=x.
\end{array}
\right.
$$

The Kob-Andersen model in $\Lambda$ with parameter $k$, for short the
KA-$k$f model, with
\emph{constrained exchanges} in $\Lambda$ and \emph{unconstrained sources} at the
boundary $\partial \Lambda$ is the continuous time Markov process defined through the generator which acts on local functions $f:\Omega_{\Lambda}\to\mathbb R$ as
\begin{equation}\label{eq:gen}\mathcal L f(\eta)=\sumtwo{x,y\in \Lambda}{\|x-y\|_1=1}c_{xy}(\eta)[f(\eta^{xy})-f(\eta)]+\sum_{x\in\partial\Lambda}[(1-\eta_x)(1-q)+\eta_x q][f(\eta^x)-f(\eta)].\end{equation}
In words, every pair of nearest neighbours sites $x,y$ such that
$c_{xy}(\eta)=1,$ with rate one and independently across the lattice,
exchange their states $\eta_x,\eta_y$. In the sequel we will sometimes
refer to such a move as a \emph{legal exchange}.
Furthermore every
boundary site, with rate one and independently from anything else,
updates its state by sampling it from the Bernoulli(1-q)
measure. Notice that these latter moves are unconstrained and that for $k=1$ the KA-$1$f chain coincides with the
symmetric simple exclusion in $\L$ with sources at $\partial \L$.
It is easy to check that the KA-$k$f chain is reversible w.r.t
$\mu_\Lambda$ and irreducible thanks to the boundary sources. Let
$T_{\rm rel}(q,L)$ be its \emph{relaxation time} i.e. the inverse of
the spectral gap in the spectrum of its generator $\cL_\Lambda$ (see e.g. \cite{Levin-2008}).
\begin{maintheorem}
\label{maintheorem:1}  
For any $q\in (0,1)$ there exist two constants $C_+(q),C_-(q)$ such that
\[
  C_-(q) L^2\le T_{\rm rel}(q,L)\le C_+(q)L^2.
\]
Moreover, as $q\to 0$ the constants $C_\pm(q)$ can be taken equal to
\begin{equation}
  \label{eq:main}
  C_+(q)=
  \begin{cases}
    \exp_{(k-1)}\big(c/q^{1/(d-k+1)}\big) &\text{ if $3\le k\le d$},\\
    \exp(c \log(q)^2/q^{1/(d-1)}) & \text{if $k=2\le d$},
  \end{cases}
\end{equation}
\begin{equation}
  \label{eq:Cminus}
  C_-(q)=
  \begin{cases}
    \exp_{(k-1)}\big(c'/q^{1/(d-k+1)}\big) &\text{ if $3\le k\le d$},\\
    \exp(c'/q^{1/(d-1)}) & \text{if $k=2\le d$},
  \end{cases}
\end{equation}
where $\exp_{(r)}$ denotes the $r$-times iterated exponential and $c,c'$
are a numerical constants.
\end{maintheorem}

\section{Proof of the upper bound in Theorem \ref{maintheorem:1}}
\label{sec:core}The standard variational characterisation of
the spectral gap of $\cL$ (see e.g. \cite{Levin-2008}) implies 
immediately that the \emph{upper bound} on $T_{\rm rel}(q,L)$ of Theorem \ref{maintheorem:1} is
equivalent to the Poincar\'e inequality
\begin{equation}
  \label{eq:3}
  \var(f)\le C(q)L^2\cD(f)\quad \forall\ f:\Omega_\Lambda\mapsto \bbR,
\end{equation}
where $C(q)$ is as \eqref{eq:main}. Above $\var(f)$ denotes the variance of $f$ w.r.t. the reversible
measure $\mu$ and $\cD(f)$ is the Dirichlet form associated to the generator \eqref{eq:gen}
\begin{equation}
  \label{eq:Diri}
  \cD(f)=
  \sumtwo{x,y\in \Lambda}{\|x-y\|_1=1}\mu\Big(c_{xy}(\nabla_{xy}f)^2\Big)
  +\sum_{x\in \partial \Lambda}\mu\Big(\var_x(f)\Big),
\end{equation}
where  $(\nabla_{xy}f)(\eta):= f(\eta^{xy})-f(\eta)$ and
 $\var_x(f)$ is the local variance w.r.t. $\eta_x$, \ie  the variance conditioned on $\{\eta_y\}_{y\neq x}$.
\begin{remark}\label{rem:gapmonotonicity}
Consider two systems with sizes $L<L'$, and let $\gamma, \gamma'$ be the spectral gaps associated with the two generators. Then
\begin{equation}\label{eq:gapmonotonicity}
\gamma' \le (2d+1)\g.
\end{equation}
To see that, let $\L = \L_L$, $\L'=\L_{L'}$, and let $\cD_\L,\cD_{\L'}$ be the
Dirichlet forms of the dynamics in $\L,\L'$. Take a function
$f:\O_{\L'} \mapsto \bbR$ depending only on the variables in $\L$ and observe that $\var_{\L'}(f)=\var_{\L}(f)$. Next we bound $\cD_{\L'}(f)$ as:
\begin{align*}
\cD_{\L'}(f)&= \sumtwo{x,y\in \L'}{\|x-y\|_1=1}\mu\left(c_{xy}(\nabla_{xy}f)^2\right)
  +\sum_{x\in \partial \L'\cap \, \partial \L}\mu \left( \var_x(f) \right) \\
  &\le \cD_{\L}(f) + \sumtwo{x\in \partial \L, y\in \L'\setminus \L}{\|x-y\|_1=1}
  \mu\left(c_{xy}(\nabla_{xy}f)^2\right)\\
  &\le \cD_{\L}(f) + \sumtwo{x\in \partial \L, y\in \L'\setminus \L}{\|x-y\|_1=1}
  \mu\left( 2 \var_x(f) \right) \le (2 d+1)\, \cD_\L(f).
\end{align*}
The last line follows because $f$ does not depend on $\eta_y$ for
$y\notin \L$ and therefore an exchange for the pair $xy, x\in \L$ and $y\in \L'\setminus \L,$ is equivalent to a spin flip at $x$. Thus,
\[
  \var_\L(f) =\var_{\L'}(f)\le \frac{1}{\g'}\cD_{\L'}(f)\le \frac{2d+1}{\gamma'}\, \cD_\L(f)
\]
implying equation \eqref{eq:gapmonotonicity}.
\end{remark}

We will prove \eqref{eq:3} in several steps.
The first step consists in proving a coarse-grained constrained
Poincar\'e inequality with long range constraints (see Proposition
\ref{prop:longrange1}) under the assumption that the probability
$\pi_\ell(k,d)$ of a
certain \emph{good} event (see Definition \ref{def:good events}) is sufficiently
large. Here $\ell$ is the mesoscopic scale
characterising the coarse-grained construction and $2\le k\le d$ is
the parameter of the KA-model. The necessary tools for this part are developed
in Sections \ref{sec:coarse-graining} and \ref{sec:tools}.

The second 
step (see Section \ref{sec:paths}) consists developing canonical
flows techniques (see e.g. \cite{Levin-2008}*{Chapter 13.5}) for the KA
model in order
to bound from above the r.h.s. of the coarse-grained
Poincar\'e inequality by $C(\ell,q)(L/\ell)^2\,\cD(f),$ with
$C(\ell,q)\le e^{O(\ell^{d-1}(|\log
  (q)|+\log(\ell)))}$ (see Proposition \ref{prop:long-vs-KA} and
Corollary \ref{cor:final proof}).

The final step (see Section \ref{sec:proof}) proves that it is
possible to choose $\ell=\ell(q,k,d)$ in such a way that
$\pi_\ell(d,k)$ is large enough and
$
C(\ell,q)\le C(q) 
$
as $q\to 0$, where $C(q)$ is as in \eqref{eq:main}.

\subsection{Coarse graining}
\label{sec:coarse-graining} Let $\ell \in \bbN$ to be fixed later on. By Remark \ref{rem:gapmonotonicity}, we may assume that $N:=L/\ell$ satisfies
$N=100^n, n\in \bbN,$ so that, in particular, $\frac{1}{2} \sqrt{N}\in \bbN$. Later on (see Section
\ref{sec:proof}) we will choose $\ell$ as a function of $q$ and
suitably diverging as $q\to 0.$
We will then consider the coarse grained lattice of boxes with side $\ell$.
These boxes will be of the form $B_i=\ell i + [\ell]^d$ for $i\in
\bbZ^d$. In order to distinguish between the standard lattice and the coarse-grained lattice we denote the latter by $\bbZ^d_\ell.$ Vertices of the
original lattice will always be called sites and they will be denoted
using the letters $x,y,\dots$ while the vertices of $\bbZ^d_\ell$ will
represent boxes and they will always be denoted using the letters $i,j,\dots$. 

For $x\in\bbZ^d$ we let $B(x):=B_{i(x)}$, where $i(x)\in
\bbZ^d_\ell$ is such that $B_{i(x)}\ni x$. We also define
$\L_\ell=[N]^d\subset \bbZ^d_\ell$ so that,  
in particular, $\L=\cup_{i\in
  \L_\ell}B_i$. Sometimes we shall simply write ``the box
$i$'' meaning the box $B_i$ and whenever we shall refer to a ``box''
it will be a generic box $i$.
 

\begin{definition}[Slice]\label{def:slice}
Let $E$ be a subset of the standard basis with size $|E| \le d-1$,
$V\subset \bbZ^d$ a set of sites, and fix a site $x\in V$. Then the $|E|$-dimensional \emph{slice} of $V$ passing through $x$ in the directions of $E$ is defined as $V \cap (x+\mathrm{span}E)$, where $\mathrm{span}E$ is the linear span of $E$.
\end{definition}

\begin{definition}[Frameable configurations]\ 
 \label{def:frameable}
Given the $d$-dimensional cube $\cC_n=[n]^d\subset \bbZ^d$ and an integer $j\le d$ we
define the $j^{th}$-frame of $\cC_n$ as the union of all $(j-1)$-dimensional
slices of $\cC_n$ passing through $(1,\dots,1)$. Next we
introduce the set of $(d,j)$-\emph{frameable} configurations of
$\{0,1\}^{\cC_n}$ as those configurations which are connected by legal
KA-$j$f exchanges inside $\cC_n$ to a configuration for which the $j^{th}$-frame of
$\cC_n$ is empty.  
\end{definition}

We are finally ready for our definition of
a box $B_i$ being good for a given configuration.

\begin{definition}[Good boxes]\ 
 \label{def:good_box}
	Given $\eta\in \O_\L,$ we say that the box $B$ is $(d,k)$-good for
        $\eta$  if all $(d-1)$-dimensional slices of $B$ are $(d-1,k-1)$-frameable for all configurations
        $\eta^{\prime}\in \O_\L$ that differ from $\eta$
        in at most one site.
The probability that the $d$-dimensional box $B$ is $(d,k)$-good will be denoted by $\pi_\ell(d,k)$.
\end{definition}
\begin{remark}
For $d=2$, $k=2$ a box is $(2,2)$-good if it contains  at least two empty sites in every row and every column.
\end{remark}
{\bf Notation warning} Whenever the value of $d,k$ is clear from
the context we shall simply write
that a box is good if it is $(d,k)$-good. We shall also say that a vertex $i\in \bbZ^d_\ell$ is $(d,k)$-good  if the box $B_i$ is $(d,k)$-good. 

\subsection{Tools from oriented percolation}
\label{sec:tools} In this section we collect and prove certain technical results from
oriented percolation which will be crucial to prove the aforementioned
coarse-grained constrained Poincar\'e inequality. We shall work on the
coarse-grained lattice $\bbZ^d_\ell$ so that any vertex $i\in
\bbZ^d_\ell$ is representative of the mesoscopic box $B_i$ in the
original lattice $\bbZ^d$. {Given a configuration $\eta\in \O$ we shall
consider the induced subgraph of $\bbZ^d_\ell$ whose vertices are
the representatives of the good boxes for $\eta$. In other words, under
the measure $\mu,$ we declare each vertex of $\bbZ^d_\ell$ \emph{good}
with probability $\pi_\ell(d,k)$ and \emph{bad} with probability
$1-\pi_\ell(d,k),$ 
independently of all the other vertices. We shall study certain oriented
percolation features of the random subgraph of $\bbZ^d_\ell$ consisting of the good
vertices when $\pi_\ell$ is sufficiently large and the main result here
is Proposition \ref{prop:prob_of_hard_crossing}. Throughout this
section the parameters $d,k$ will be kept fixed and we shall write
$\pi_\ell:=\pi_\ell(d,k)$. In the sequel, and up to Proposition
\ref{prop:prob_of_hard_crossing}, we shall assume that a partition of
the vertices of $\bbZ^d_\ell$ into good and bad ones has been given.} 
\begin{definition}[Paths]
An \emph{up-right} or \emph{oriented} path $\gamma$ in $\bbZ^d_\ell$
starting at $i$ and of length $n\in \bbN$ is a sequence
$\left(\gamma^{(1)},\dots,\gamma^{(n)}\right)\subset
\bbZ^d_\ell$ 
such that $\gamma^{(1)}=i$ and $\gamma^{(t+1)}\in\{
  \gamma^{(t)}+\vec{e}_{1}, \gamma^{(t)} + \vec{e}_{2}\} $ for
all $t\in [n-1]$. 
$\gamma$ is \emph{focused} if 
$
d_\gamma(t):=d\big(\gamma^{(t)},\{j:\
j=i+s(\vec e_1+\vec e_2),\ s\in \bbN\}\big)
$ satisfies $\max_{t\in [n]}d_\gamma(t)\le \sqrt{n}$. Two consecutive elements of $\gamma$ form an
edge of $\gamma$ and we say that $\gamma,\gamma'$ are \emph{edge-disjoint} if they do not share
an edge. We say that $\gamma$ is \emph{good}
if
$\gamma^{(t)}$ is good for all $t\in [n]$. 
\end{definition}
\begin{definition}[Good family of paths]
\label{def:good family}Fix $i\in\bbZ^d_\ell$. A family of paths $\cG$ is
said to form a good family for $i$ if the following conditions hold:

\begin{enumerate}
\item All paths in $\cG$ are good up-right focused paths
  starting at $i$ of length $2N$.
\item The paths of $\cG$ are \emph{almost edge-disjoint} \ie any
  common edge is at distance at most $\sqrt{N}$ from $i$.
\item $\left|\cG\right| \ge \frac 12 \sqrt{N}$.
\end{enumerate}
\end{definition}
\begin{remark}
There are $\sqrt{N}$ sites at distance $\sqrt{N}$ from $i$ that can be
reached by an up-right path (recall that distances are in the
$\ell_1$-norm). Consider now the paths of $\cG$ passing through a
vertex $j$ whose distance from $i$ is $\sqrt{N}$. Such a path could go
either up or to the right, and since these edges are at distance
larger than $\sqrt{N}$ from $i$, (2) in the above definition implies
that only two such paths could exist. Therefore, $\left|\cG\right|\le 2\sqrt{N}$. 
\end{remark}
Given $i\in \bbZ^d_\ell$ let $D_i$ be the segment 

\begin{equation}
  \label{eq:Di}
D_i=
\Big\{
i+\Big(\frac{1}{2}\sqrt{N}-t\Big)\vec{e}_{1}+\big(\frac{1}{2}\sqrt{N}+t\big)\vec{e}_{2}\,|\,-\frac{1}{2}\sqrt{N}\le
t\le\frac{1}{2}\sqrt{N}\Big\}.
\end{equation}

Let also 
$H_n,V_n$ be the rectangular subsets of the form 
\begin{align*}
H_n& = \{j\in \bbZ_\ell^d:\ j= i+a\vec e_1 +
b\vec e_2, a\in [0,\ell_n], b\in [0,\ell_{n-1}]\}\\
V_n&= \{j\in \bbZ_\ell^d:\ j= i+a\vec e_1 +
b\vec e_2, a\in [0,\ell_{n-1}], b\in [0,\ell_{n}]\}
\end{align*}
where $\ell_n=10^n$.  We shall prove that the existence of a good
family of paths for the vertex $i\in \bbZ_\ell^d$ is
guaranteed by the simultaneous occurrence of certain events $\cA,\cB$
and $\{\cC_n\}_{n=1}^{n_*},$ where, recalling the choice of $L$ in the beginning of Section \ref{sec:coarse-graining}, $n_*$ is such that $\ell_{n^*}=\sqrt{N}$ (cf. Figure \ref{fig:crossings}).

\begin{figure}[!ht]
\begin{tikzpicture}[scale=0.3, every node/.style={scale=1}]

	\draw[gray, very thin] (0,0)--(0,3)--(1,3)--(1,0)--(0,0);
	\draw[gray, very thin] (0,0)--(3,0)--(3,1)--(0,1)--(0,0);
	\draw[gray, very thin] (0,0)--(0,9)--(3,9)--(3,0)--(0,0);
	\draw[gray, very thin] (0,0)--(9,0)--(9,3)--(0,3)--(0,0);
	\draw[gray, very thin] (0,0)--(0,27)--(9,27)--(9,0)--(0,0);
	\draw[gray, very thin] (0,0)--(27,0)--(27,9)--(0,9)--(0,0);
	
	\draw (-1,1) node[below,black] {\large $i$};
	
	\draw[line width=0.85mm, decorate,decoration=snake,segment length=30pt] (0,0.3)--(2.3,0.7);
	\draw[decorate,decoration=snake,segment length=30pt] (2.3,0.7)--(3,0.8);

	\draw[decorate,decoration=snake,segment length=30pt] (0.2,0)--(0.9,3);
	
	\draw[decorate,decoration=snake,segment length=40pt] (0,1.1)--(9,2.5);
	\draw[decorate,decoration=snake,segment length=40pt] (0.9,0)--(2.8,9);
	
	\draw[decorate,decoration=snake,segment length=50pt] (0,3.1)--(27,8);
	
	\draw[decorate,decoration=snake,segment length=22pt] (2.2,0) to (2.3,0.7);
	\draw[line width=0.85mm,decorate,decoration=snake,segment length=45pt] (2.3,0.7) to (4,7.6);
	\draw[decorate,decoration=snake,segment length=40pt] (4,7.6) to (7.5,27);

	\draw[blue, very thin] (0,27)--(27,0);
	
	\draw[BrickRed,decorate,decoration=snake,segment length=50pt] (0,4)--(4,7.6);
	\draw[line width=0.85mm, BrickRed,decorate,decoration=snake,segment length=50pt] (4,7.6)--(11.5,15.5);
	
	\draw[BrickRed,decorate,decoration=snake,segment length=50pt] (4,0)--(15.5,11.5);
	
	\draw[BrickRed,decorate,decoration=snake,segment length=50pt] (0,7)--(10,17);
	
\draw[<-] (0,-1)--(10,-1);
\draw[->] (17,-1)--(27,-1);
\draw (13.5,0) node[below,black] {$\ell_{n^*}=\sqrt{N}$};

\end{tikzpicture}

\caption{\label{fig:crossings} A graphical illustration of the proof
  of Lemma \ref{lem:conditions_for_existance_of_good_family}. For better rendering the drawn
  paths are not perfectly oriented and the ratio among the sides of rectangles in the drawings is not $1/10$ as it should be. 
  The blue segment corresponds to the set $D_i$. The red paths are the good up-right paths  guaranteed by the event $\mathcal B$. The blacks paths are the good up-right hard crossings guaranteed by the events $C_n$.}
\end{figure}
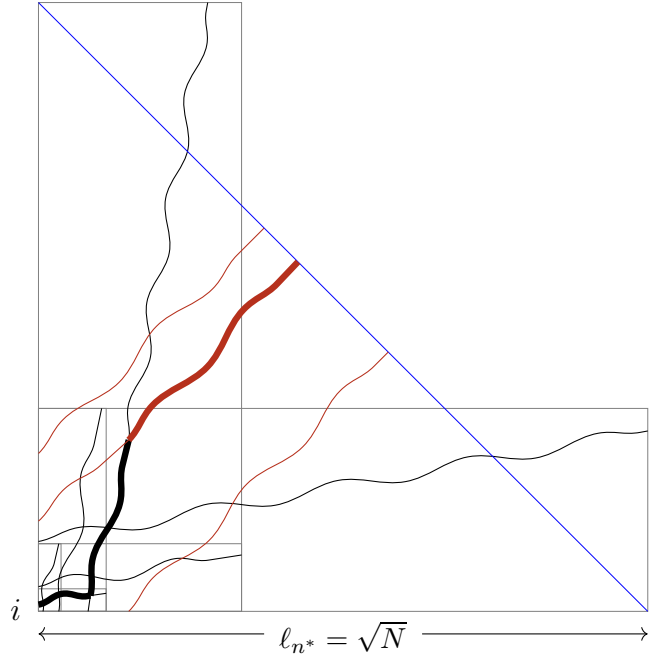

\begin{definition}[The events $\mathcal A$, $\mathcal B$ and $\mathcal C_n$]
  \label{def:good events}\ 
\begin{enumerate}[(i)]
\item Let $R_i$ be the rectangle in $\bbZ^d_\ell$ whose short sides
  are $D_i$ and $D_i+2N(\vec e_1+\vec e_2)$. Then $\cA$ is the
event that there are at least $1.9\sqrt{N}$ edge-disjoint good up-right
paths contained in $R_i$ and connecting $D_i$ with $D_i+2N(\vec e_1+\vec e_2)$.
\item $\cB$ is the event that the set
$\cup_{t\in [0,\sqrt{N}\,]}\{i+t\vec e_1\}\cup\{i+t\vec e_2\}$ is connected to at least $0.7\sqrt{N}$ vertices of $D_i\setminus(H_{n_*}\cup V_{n_*})$
by a good up-right path,
\item $\cC_n$ is the event that $i$ is good and there exists a good up-right hard-crossing of
  both $V_{n}$ and $H_{n}$, \ie a good up-right path connecting
  the two short sides of $V_n(H_n)$ and which is contained in $V_n(H_n)$.
\end{enumerate}
\end{definition}
The next lemma guarantees the existence of a good family of paths for
$i\in \bbZ^d_\ell$. We emphasise that these are paths whose vertices
represent good boxes.
\begin{lemma}
\label{lem:conditions_for_existance_of_good_family}
Assume that $\cA\cap \cB\cap \cC_n$ occurs for all $n\in [n_*]$. 
Then there exists a good family of paths for $i\in \bbZ^d_\ell$.  
\end{lemma}
\begin{proof}
We show first that $i$ is connected by a good up-right path to the set 
$D_i\cap H_{n_*}$ and to the set $D_i\cap V_{n_*}$. Let $n_1=1$, and define recursively $n_{k+1}, k\in[n_\star-1]$ as the largest integer 
$n\in[n_*]$ such that there exists a crossing of $H_{n}$
starting from the set $\{i+t \vec{e}_2, t\in[0,\ell_{n_k}]\}$. Since the events $\cC_n$ all occur, the sequence $\{n_k\}_{k=1} ^n$ is strictly increasing as long as $n_k \le n_*$.

Then, starting from $V_{1}\equiv V_{n_1}$ we can first follows the lowest hard crossing of $H_{n_2}$ until we reach a hard crossing of
$V_{n_2}$. Then we follow the latter until meeting a hard crossing of $H_{n_3}$ and so on.
At the end of this procedure the set $V_{1}$, and a fortiori the box $i$, becomes connected by a good up-right path to the right short side of $H_{n_*}$ and hence also to  one of the vertices of $D_i\cap H_{n_*}$.
The same construction can be repeated symmetrically by inverting the role of $H$ and $V$. Therefore we conclude that there exists a good up-right path connecting $i$ to $D_i\cap H_{n_*}$ and  a good up-right path connecting $i$ to $D_i\cap V_{n_*}$. See Figure \ref{fig:crossings}.

Recall that $\ell_{n_*}=\sqrt {N}$. Then, since $|D_i\cap H_{n_*}|=|D_i\cap V_{n_*}|=\sqrt N/10$ and since each each vertex can be the starting point of at most two edge disjoint
paths, there could be at most $2|(D_i\cap H_{n_*}) \cup (D_i\cap V_{n_*})|=0.4 \sqrt{N}$ edge-disjoint path starting from $(D_i\cap H_{n_*}) \cup (D_i\cap V_{n_*})$. Hence, the event $\cA$ guarantees that there are at least $1.9 \sqrt{N} - 0.4\sqrt{N} = 1.5 \sqrt{N}$ edge-disjoint paths starting in $D_i\setminus (H_{n_*}\cup V_{n_*})$. Since each starting point for these paths could belong to at most two paths, at least $0.75 \sqrt{N}$ boxes of $D_i\setminus (H_{n_*}\cup V_{n_*})$ are the starting point of an edge-disjoint up-right paths crossing $R_i$.

Using event $\mathcal B$ and noticing that $|D_i\setminus (H_{n_*}\cup V_{n_*})|=0.8 \sqrt N$, at most $0.1 \sqrt{N}$ boxes in $D_i\setminus (H_{n_*}\cup V_{n_*})$ are not connected to 
$\cup_{t\in [0,\sqrt{N}\,]}\{i+t\vec e_1\}\cup\{i+t\vec e_2\}$. That is, the number of boxes in $D_i\setminus (H_{n_*}\cup V_{n_*})$ that are at the same time the starting point of an edge-disjoint up-right path crossing $R_i$ and connected to $\cup_{t\in [0,\sqrt{N}\,]}\{i+t\vec e_1\}\cup\{i+t\vec e_2\}$ is at least $0.75 \sqrt{N} - 0.1 \sqrt{N}=0.65\sqrt{N}$. 

Using now the fact that $i$ is connected by a good up-right path to the set
$D_i\cap H_{n_*}$ and to the set $D_i\cap V_{n_*}$, 
we conclude that there
exist at least $0.65\sqrt{N}$ good up-right paths from $i$ to
$D_i+2N(\vec e_1+\vec e_2)$ which, after crossing $D_i$ become
edge-disjoint and never leave $R_i$. The  thick path of Fig. \ref{fig:crossings} is one of these paths, drawn up to its crossing with $D_i$ ($R_i$ is not depicted in the figure due to lack of space). 
These paths form the sought good
family as required. 
\end{proof}
{Our next task is to prove that if $\pi_\ell$ is sufficiently close to
one then, uniformly in $n$, $\cA,\cB$ and $\cC_n$ are very likely.} As proved in Section \ref{sec:proof} that will be the case if the mesoscopic scale $\ell$ is suitably chosen as a function of $q,d,k$.
\begin{proposition}
\label{prop:prob_of_hard_crossing}
For any $\lambda>0$ there exists $\pi_*<1$ such that for $\pi_\ell\ge \pi_*$ and all $n,N\in \bbN$ 
\begin{align*}
  \label{eq:1}
&(a)\qquad \mu(\cC_n)\ge 1-e^{-\lambda \ell_{n-1}},\\
&(b) \qquad\mu(\cB)\ge 1-e^{-\lambda\sqrt{N}},\\
&(c) \qquad \mu(\cA)\ge 1-e^{-\lambda\sqrt{N}}.
\end{align*}
In particular a family of good paths starting at $i$ exists w.h.p if
$\pi_\ell$ is sufficiently close to one.
\end{proposition}
\begin{proof}\ 

(a) This can be proven by a contour argument. Consider the rectangle $V_{n}$, and assume that it does not contain a  good hard crossing. Then consider the path on the dual lattice that forms the upper contour of the set of sites that are connected to the bottom of the rectangle via an up-right good path. Since there is no vertical crossing, this path necessarily takes $\ell_n$ steps to the right and ends somewhere on the right boundary of $V_{n}.$ 
By using the fact that each time this dual path makes a step to the right or downwards, this implies the presence of a bad vertex, it is not difficult to prove that for $\pi_{\ell}$ sufficiently large depending on $\lambda$
$$\mu(\mbox{there is not a good hard crossing})=\mu(\mathcal C_n^c) \leq e^{-\lambda\ell_n}.$$

(b) Consider the down-left good oriented paths  starting from sites of $D_i\setminus (H_{n_*}\cup V_{n_*})$. The event $\mathcal B$ certainly occurs if  at least $7/8$ of the points in this set are the starting point of an infinite  down-left good oriented path.
The upper bound on the probability of $\cB$ then follows directly from \cite{Durrett-Schonmann}*{Theorem 1} \footnote {Though the Theorem is stated for the contact process, it also holds for oriented percolation as stated in \cite
{Durrett-Schonmann}).}.

 (c)
 The main tool here is the \emph{max-flow min-cut theorem} (see
 e.g. \cite{Bollobas}). Given a \emph{directed} graph
 $\left(V,E\right)$ a \emph{flow} $f$ is a non-negative
function defined on the edges; we write $f(u,v)$ instead of $f(\overrightarrow
{uv})$ for the value of the
flow on the directed edge $\overrightarrow{uv}$. Given two disjoint sets of
vertices $s=\{s_1,\dots,s_k\}, t=\{t_1,\dots, t_m\}$ called the \emph{sources}
and the \emph{sinks} respectively we say that $f:E\to\mathbb
R^+$ is a flow from $s$
to $t$ if for all $v\not \in \{s_1,\dots,s_k\}\cup \{t_1,\dots, t_m\}$
\[
\sum_{u:(u,v)\in E} f(u,v)=\sum_{w:(v,w)\in E} f(v,w).
\]
In other words,
 for all vertices outside $s\cup t$ the incoming flow equals
the outgoing flow. Finally, given a \emph{capacity}
function $c:E\to \bbR^+$ we say that a flow $f$ satisfies the capacity constraint if $f(e)\le
c(e)$ for all $e\in E$. 
Given a flow from $s$
to $t$ the {\sl value of the flow } $v(f)$ is defined
as the total flow going in the sinks (which is the same
as the flow leaving the sources), namely 
\[
v(f):=\sum_{j=1}^m\sum_{v:(v,t_j)\in E}f(v,t_j).
\]
A \emph{cut} $\left(S,T\right)$ is a partition of $V$ in two subsets
$S$ and $T$, such that all the sources belong to $S$ and all the sinks
belong to $T$. The capacity of a cut $(S,T)$ is the sum of capacities of the edges pointing
from $S$ to $T$.
\begin{theorem*}[Max-Flow Min-Cut theorem] Given a set of sources $s$ and
  a set of sinks $t$ and a capacity function $c$, the
  maximum value of a flow from $s$ to $t$ satisfying the capacity constraint is equal the minimum
capacity of a cut. Moreover, if all capacities are integers,  there is
a flow $f$ satisfying the above requirements such that $f(e)\in\bbN$
for all $e\in E$ and its value $v(f)$ is maximal.
\end{theorem*}
In order to use this theorem for the proof of part (c) of the
proposition we first define our graph $G=(V,E)$. The vertex
set is
\[
V=\left\{ i+a\vec{e}_{1}+b\vec{e}_{2}\,:\,a,b\in\left[N\right],\,a+b\ge\sqrt{N},\,\left|a-b\right|\le\sqrt{N}\right\} \cap\Lambda_\ell,
\]
and the directed edges are
\[
E=\left\{ \left(j,j^{\prime}\right)\,:\,j\mbox{ is good and }j^{\prime}\in\left\{ j+\vec{e}_{1},j+\vec{e}_{2}\right\} \right\} .
\]
\begin{remark}
Notice that the requirement that edges have their starting point only at
the good vertices of $V$ is the only place where randomness enters.
\end{remark}
We then choose as source set the set:
\[
s=\left\{ j\in V\,:\,\|i-j\|_{1}=\sqrt{N}\right\} ,
\]
and as sink set the set:
\[
t=V\cap\left\{ j\,|\,\left(\vec{e}_{1},j\right)=N\mbox{ or }\left(\vec{e}_{2},j\right)=N\right\} .
\]
Finally we assign unitary capacity to all edges of $E$. With this
choice, the maximal value of a flow $f$ from $s$ to $t$ satisfying
$f(e)\in \{0,1\}$ will be exactly 
the number of edge-disjoint good up-right
paths contained in $R_i$ and connecting $D_i$ with $D_i+2N(\vec e_1+\vec e_2)$. We have thus
reduced the proof of part (c) to the following claim:
\begin{claim}
If $\pi_\ell$ is large enough, with probability greater than
$1-e^{-\lambda\sqrt{N}}$ the graph constructed above is such that
the capacity of any cut is at least $1.9\sqrt{N}$.
\end{claim}
\begin{proof}[Proof of the claim]
In order to prove the claim, for every cut
$\left(S,T\right)$ we will construct a dual path $\g:=\gamma_{S,T}^{*}$ that will separate
$S$ from $T$ satisfying the following property. If the capacity of the cut is
smaller than $1.9\sqrt{N}$ then there are at least
$|\g|/2-0.9\sqrt{N}$ vertices in $V$ which are bad and which neighbor
$\g$. Here $|\g|\ge 2\sqrt{N}$ is the length of $\g$. A simple Peierls argument then proves that the latter event has
probability at most $e^{-\l \sqrt{N}}$ for any $\pi_\ell$ large enough.

First, let us define a dual graph $V^{*}$ for some fixed $\left(S,T\right)$.
Its vertices will be the faces of $\Lambda_\ell$ that have at least three
neighbors in $V$. That is,
\[
V^{*}=\left\{ i^{*}\in\Lambda_\ell+\frac{1}{2}\vec{e}_{1}+\frac{1}{2}\vec{e}_{2}\, : \,\#\left\{ i\in V\,:\,\|i^{*}-i\|_{1}=1\right\} \ge3\right\} .
\]

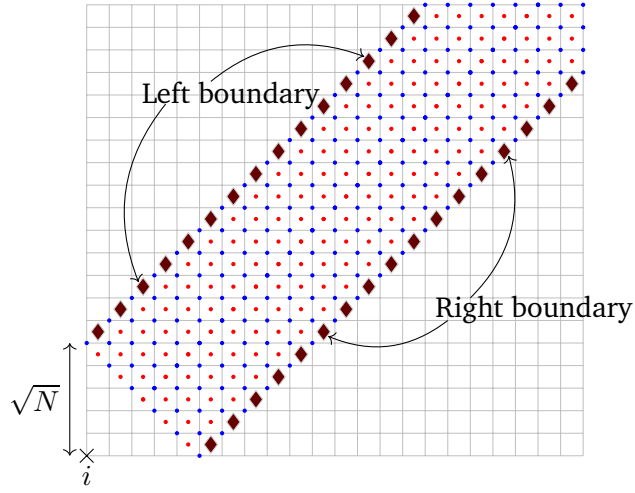
\begin{figure}[ht]
    \begin{tikzpicture}[scale=0.3]
      
		\draw[step=1, lightgray, very thin] (0,0) grid (22,20);
		\draw (0,0) node[cross=3pt,black] {};
		\draw (0,0) node[below,black] {$i$};
      	\foreach \d in {3,...,30} {
      		\ifnum \d < 21
      			\filldraw[blue] (\d,\d) circle (2pt);
      		\fi
      		\foreach \x in {0,...,2} {
      			\ifnum \numexpr\d+\x<23
      				\ifnum \numexpr \d-\x<21
      					\filldraw[blue] (\d+\x,\d-\x) circle (2pt);
      				\fi
      				\ifnum\numexpr\d-\x > 0
      					\ifnum \numexpr \d-\x<22
      						\filldraw[blue] (\d+\x,\d-\x-1) circle (2pt);
      					\fi
      				\fi
      			\fi
      			
				\ifnum \numexpr\d+\x<21
      				\ifnum \numexpr \d-\x<24
      					\filldraw[blue] (\d-\x,\d+\x) circle (2pt);
      				\fi
      				\ifnum\numexpr\d-\x > 0
      					\ifnum \numexpr \d-\x<21
      						\filldraw[blue] (\d-\x-1,\d+\x) circle (2pt);
      					\fi
      				\fi
      			\fi      			
      		}
      	}   
      	
      	\foreach \d in {3,...,30} {
      		\foreach \x in {0,...,2} {
      			\ifnum \numexpr\d+\x<22
      				\ifnum \numexpr \d-\x<21
      					\filldraw[red] (\d+\x+0.5,\d-\x-0.5) circle (2pt);
      				\fi
      			\fi
      			\ifnum\numexpr\d-\x > 0
      				\ifnum \numexpr \d+\x<23
      					\ifnum \numexpr \d-\x < 21
      						\filldraw[red] (\d+\x-0.5,\d-\x-0.5) circle (2pt);
      					\fi
      				\fi
      			\fi
      			
				\ifnum \numexpr\d+\x<20
      				\ifnum \numexpr \d-\x<23
      					\filldraw[red] (\d-\x-0.5,\d+\x+0.5) circle (2pt);
      				\fi
      			\fi
      			\ifnum\numexpr\d-\x > 0
      				\ifnum \numexpr \d+\x<21
      					\ifnum \numexpr \d-\x < 22
      						\filldraw[red] (\d-\x-0.5,\d+\x-0.5) circle (2pt);
      					\fi
      				\fi
      			\fi    			
      		}    		
      	}
      	
      	\foreach \d in {0,...,16} {
      		\node[draw=black!20,diamond,aspect=0.75,scale=0.4, fill=black!60!red] (nn) at (\d+5.5,\d+0.5){};

      	}
      	\foreach \d in {0,...,14} {
      	      	\node[draw=black!20,diamond,aspect=0.75,scale=0.4, fill=black!60!red] (nn) at (\d+0.5,\d+5.5){};
      	}
      	
      	\draw[<->] (-0.75,0) -- (-0.75,5);
      	\draw (-0.75,2.5) node[left,black] {$\sqrt{N}$};
      	
      	\draw (15,6.5) node[right,black] {Right boundary};
      	\draw[<-] (10.7,5.3) [bend right] to (16,6);
      	\draw[<-] (18.7,13.3) [bend left] to (17,7);
      	
      	\draw (2,16) node[right,black] {Left boundary};
      	\draw[<-] (2.3,7.7) [bend left] to (3.5,15.6);
      	\draw[<-] (12.3,17.7) [bend right] to (5.5,16.6);
    \end{tikzpicture}
\caption{Black dots are the vertices of $V$, grey dots are the vertices of $V^{*}$, diamonds are the left and right boundary of $V^{*}$.}

\end{figure}

Its (directed) edges will depend on the cut $\left(S,T\right)$. For
$i^{*},j^{*}\in V^{*}$, $\left(i^{*},j^{*}\right)$ is an edge if
$\|i^{*}-j^{*}\|_{1}=1$, and if it has a site of $S$ to its left
and a site of $T$ to its right. We will separate the vertices of
$V^{*}$ in three parts:
\begin{enumerate}
\item The right boundary
\[
\left\{ i+\left(\sqrt{N}+\frac{1}{2}+a\right)\vec{e}_{1}+\left(\frac{1}{2}+a\right)\vec{e}_{2}\,:\,a\in\left[N\right]\right\} \cap V^{*},
\]
\item the left boundary
\[
\left\{ i+\left(\sqrt{N}+\frac{1}{2}+a\right)\vec{e}_{1}+\left(\frac{1}{2}+a\right)\vec{e}_{2}\,:\,a\in\left[N\right]\right\} \cap V^{*},
\]
\item the interior, which will include all vertices that are neither in
the right nor in the left boundary.
\end{enumerate}
Focusing on a fixed vertex $j^{*}\in V^{*}$, we can count the edges
going into $j^{*}$ and the edges going out of $j^{*}$ if we know
which of the neighbouring vertices of $V$ (namely $\left\{ j\in V\,:\,\|j^{*}-i\|_{1}=1\right\} $)
are in $S$ and which are in $T$. By checking all possibilities,
one can verify that the incoming degree of a vertex in the interior
of $V^{*}$ equals its outgoing degree (see right part of Fig. \ref{fig:counting}).
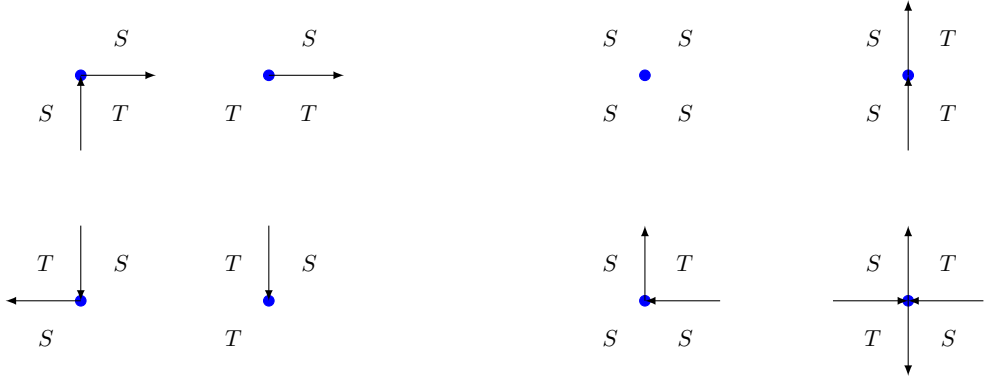
\begin{figure}[ht]
\begin{tikzpicture}[scale=0.5, every node/.style={scale=0.8}]
    \def \x{0}
    \def \y{3}	
    
	\filldraw[blue] (\x-0.5,\y) circle (4pt);
	\draw (\x+1,\y+1) node[left,black] {$S$};
	\draw (\x-1,\y-1) node[left,black] {$S$};
	\draw (\x+1,\y-1) node[left,black] {$T$};
	\draw[-latex]  (\x-0.5,\y) to +(2,0);
	\draw[-latex]  (\x-0.5,\y-2) to +(0,2);

	\def \x{5}
    \def \y{3}	
    	
	\filldraw[blue] (\x-0.5,\y) circle (4pt);
	\draw (\x+1,\y+1) node[left,black] {$S$};
	\draw (\x-1,\y-1) node[left,black] {$T$};
	\draw (\x+1,\y-1) node[left,black] {$T$};
	\draw[-latex]  (\x-0.5,\y) to +(2,0);
	
	\def \x{0}
    \def \y{-3}	
    
	\filldraw[blue] (\x-0.5,\y) circle (4pt);
	\draw (\x+1,\y+1) node[left,black] {$S$};
	\draw (\x-1,\y-1) node[left,black] {$S$};
	\draw (\x-1,\y+1) node[left,black] {$T$};
	\draw[-latex]  (\x-0.5,\y) to +(-2,0);
	\draw[-latex]  (\x-0.5,\y+2) to +(0,-2);

	\def \x{5}
    \def \y{-3}	
    	
	\filldraw[blue] (\x-0.5,\y) circle (4pt);
	\draw (\x+1,\y+1) node[left,black] {$S$};
	\draw (\x-1,\y-1) node[left,black] {$T$};
	\draw (\x-1,\y+1) node[left,black] {$T$};
	\draw[-latex]  (\x-0.5,\y+2) to +(0,-2);

	\def \x{15}
    \def \y{3}	
    	
	\filldraw[blue] (\x-0.5,\y) circle (4pt);
	\draw (\x-1,\y+1) node[left,black] {$S$};
	\draw (\x+1,\y+1) node[left,black] {$S$};
	\draw (\x+1,\y-1) node[left,black] {$S$};
	\draw (\x-1,\y-1) node[left,black] {$S$};

	\def \x{15}
    \def \y{-3}	
    	
	\filldraw[blue] (\x-0.5,\y) circle (4pt);
	\draw (\x-1,\y+1) node[left,black] {$S$};
	\draw (\x+1,\y+1) node[left,black] {$T$};
	\draw (\x+1,\y-1) node[left,black] {$S$};
	\draw (\x-1,\y-1) node[left,black] {$S$};
	\draw[-latex]  (\x+1.5,\y) to +(-2,0);
	\draw[-latex]  (\x-0.5,\y) to +(0,2);

	\def \x{22}
    \def \y{3}	
    	
	\filldraw[blue] (\x-0.5,\y) circle (4pt);
	\draw (\x-1,\y+1) node[left,black] {$S$};
	\draw (\x+1,\y+1) node[left,black] {$T$};
	\draw (\x+1,\y-1) node[left,black] {$T$};
	\draw (\x-1,\y-1) node[left,black] {$S$};
	\draw[-latex]  (\x-0.5,\y-2) to +(0,2);
	\draw[-latex]  (\x-0.5,\y) to +(0,2);

	\def \x{22}
    \def \y{-3}	
    	
	\filldraw[blue] (\x-0.5,\y) circle (4pt);
	\draw (\x-1,\y+1) node[left,black] {$S$};
	\draw (\x+1,\y+1) node[left,black] {$T$};
	\draw (\x+1,\y-1) node[left,black] {$S$};
	\draw (\x-1,\y-1) node[left,black] {$T$};
	\draw[-latex]  (\x-0.5,\y) to +(0,-2);
	\draw[-latex]  (\x-0.5,\y) to +(0,2);
	\draw[-latex]  (\x+1.5,\y) to +(-2,0);
	\draw[-latex]  (\x-2.5,\y) to +(2,0);
	
\end{tikzpicture}
\caption{\label{fig:counting}Incoming and outgoing degrees of vertices on the boundary of $V^{*}$
(left) and its interior (right)}

\end{figure}

At the boundaries, however, there could be vertices that have an outgoing
degree different from the incoming degree. Consider a site on the
right boundary (see left part of Fig. \ref{fig:counting}) $j_{a}^{*}=i+\left(\sqrt{N}+\frac{1}{2}+a\right)\vec{e}_{1}+\left(\frac{1}{2}+a\right)\vec{e}_{2}$.
Let $j_{a}^{+}=j_{a}^{*}+\frac{1}{2}\vec{e}_{1}+\frac{1}{2}\vec{e}_{2}$
and $j_{a}^{-}=j_{a}^{*}-\frac{1}{2}\vec{e}_{1}-\frac{1}{2}\vec{e}_{2}$.
If both $j_{a}^{+}$ and $j_{a}^{-}$ are in $S$, or if both are
in $T$, then the incoming degree of $j_{a}^{*}$ is the same as its
outgoing degree. However, if $j_{a}^{+}\in S$ and $j_{a}^{-}\in T$
then the incoming
degree is $1$ and the outgoing
degree is $0$.
For the case $j_{a}^{+}\in T$ and $j_{a}^{-}\in S$, we have an outgoing
degree $1$ and incoming
degree $0$. $j_{a}^{+}=j_{a+1}^{-}$, therefore
the total outgoing degree of sites on the right boundary is 
\[
\#\left\{ a\,:\,j_{a}^{-}\in S,\,j_{a+1}^{-}\in  T\right\} 
\]
and the total incoming degree is
\[
\#\left\{ a\,:\,j_{a}^{-}\in  T,\,j_{a+1}^{-}\in S\right\} .
\]
But since the first site (\ie, $j_{0}^{-}$) is in $s$ and the last
site is in $t$, the incoming degree must be smaller by $1$ than
the outgoing degree.

By the exact same argument, we can find that the incoming degree of
the left boundary is larger by $1$ than its outgoing degree. This
implies that there exists a dual path $\gamma_{*}=\left(j_{*}^{(1)},\dots,j_{*}^{(n)}\right)$,
where $j_{*}^{(1)}$ in on the right boundary and $j_{*}^{(n)}$ is
on the left boundary. In particular, $n\ge2\sqrt{N}$.

The capacity of the cut $\left(S,T\right)$ is at least the number of
edges in $E$ pointing from $S$ to $T$ and crossing $\gamma_{*}$.
Thanks to the choice of the direction of the edges in $V^{*}$, this
could be written as
\[
\#\left\{ t\,:\,j_{*}^{(t+1)}-j_{*}^{(t)}\in\left\{ -\vec{e}_{1},\vec{e}_{2}\right\} \mbox{ and }\left(j_{*}^{(t+1)},j_{*}^{(t)}\right)\mbox{ crosses an edge in }E\right\} .
\]
We therefore consider the number of steps that $\gamma_{*}$ takes
in each direction:
\begin{eqnarray*}
R & = & \#\left\{ t\,:\,j_{*}^{(t+1)}-j_{*}^{(t)}=\vec{e}_{1}\right\} ,\\
L & = & \#\left\{ t\,:\,j_{*}^{(t+1)}-j_{*}^{(t)}=-\vec{e}_{1}\right\} ,\\
U & = & \#\left\{ t\,:\,j_{*}^{(t+1)}-j_{*}^{(t)}=\vec{e}_{2}\right\} ,\\
D & = & \#\left\{ t\,:\,j_{*}^{(t+1)}-j_{*}^{(t)}=-\vec{e}_{2}\right\} .
\end{eqnarray*}
Observe that $i_{*}^{(n)}-i_{*}^{(1)}=\left(R-L\right)\vec{e}_{1}+\left(U-D\right)\vec{e}_{2}$,
and since
\begin{eqnarray*}
\left(\vec{e}_{1}-\vec{e}_{2},j_{*}^{(1)}\right) & = & \left(\vec{e}_{1}-\vec{e}_{2},i\right)+\sqrt{N},\\
\left(\vec{e}_{1}-\vec{e}_{2},j_{*}^{(n)}\right) & = & \left(\vec{e}_{1}-\vec{e}_{2},i\right)-\sqrt{N},
\end{eqnarray*}
$U+L-D-R=2\sqrt{N}$. Therefore, since $U+L+R+D=n$ we get  $U+L=\frac{n}{2}+\sqrt{N}$.
\begin{definition}
We will say that a pair $(j,j')$ of vertices of $V$ form a
\emph{erased edge} if the vertex $j$ is bad and $j'\in \left\{ j+\vec{e}_{1},j+\vec{e}_{2}\right\}$. In other words, these are the edges of the original $\Lambda_\ell$ that do not belong to our graph $G=(V,E).$
\end{definition}
Assume now that the capacity of the cut is less than $1.9\sqrt{N}$.
From the previous observations it follows that
$\gamma_{*}$ must cross at least {$U+L-1.9\sqrt N=n/2-0.9\sqrt{N}$} 
erased edges. Since
every such erased edge comes from a bad vertex, and since at most
two erased edges could share the same bad vertex, at least {$n/4-0.45\sqrt{N}$} 
of the vertices to the left of $\gamma_{*}$ are bad. Therefore,  
the probability that there exists a cut with capacity less than  $1.9\sqrt N$ is upper bounded by the probability that there exists a dual path of length $n\geq 2\sqrt N$ 
with at least $n/4-0.45\sqrt{N}$ bad vertices on its left.
Since there are at most $2^{n}$ dual paths of length $n$, if $\pi_\ell$ was taken large enough depending on $\lambda$, we get
\begin{gather*}
\mu(\text{capacity of any cut } 
\geq 1.9\sqrt N) \\
\geq 1-\sum_{n=2\sqrt{N}}^{\infty}2^{n} \sum_{k={n/4-0.45\sqrt{N}}}^n\binom{n}{k}\left(1-\pi\right)^{k}\geq 1-e^{-\lambda\sqrt N}.
\end{gather*}
\end{proof}
The proof of the proposition is complete.\end{proof}
\subsection{A long range Poincar\'e inequality}\label{sec:long}
Recall the setting of Sections \ref{sec:coarse-graining}, \ref{sec:tools}  and in particular Definition
\ref{def:good family} of a good
family of paths for a vertex $i\in
\bbZ^d_\ell.$ Let $Q_i=
i+ \{0,1\}^d \setminus \{0\}^d\subset \bbZ^d_\ell$ and define
\begin{equation}
  \label{eq:long_range_contstaint}
  \hat c_i=
  \begin{cases}
    1 & \parbox[t]{.8\textwidth}{if any $j\in Q_i$ is good and
      there exists a good family of paths for $i+\vec e_1$,}\\
0 &\text{ otherwise.} 
  \end{cases}
\end{equation}
In this section we shall prove the following result. Recall that $\pi_\ell:=\pi_\ell(d,k)$
is the probability that any given $i\in \bbZ^d_\ell$ is $(d,k)$-good.  
\begin{proposition}
\label{prop:longrange1}There exists $\pi_*<1$ such that for any $\pi_\ell \ge \pi_*$ and any
local function $f:\Omega_\L \to \bbR$
\begin{equation}
  \label{eq:fabio1}
\var(f)\le 4 \sum_{i\in \L_\ell} \mu\Bigl(\hat c_i \var_{B_i}(f)\Bigr).
\end{equation}
\end{proposition}
\begin{proof}
We will closely follow the proof of \cite{MT}*{Theorem 2.6}. Let $\tilde c_i$ be the indicator of the event $\cA\cap \cB\cap_{n=1}^{n_*}\cC_n$ (see Definition \ref{def:good events}), together with the requirement that $Q_i$ is good.
By Lemma \ref{lem:conditions_for_existance_of_good_family} $\tilde c_i\le \hat c_i$
for all $i\in \bbZ^d_\ell$.  Hence, in order to prove \eqref{eq:fabio1}, it is enough to prove the
stronger constrained Poincar\'e inequality in which in the r.h.s. of \eqref{eq:fabio1}
the constraint $\hat c_i$ is replaced by $\tilde c_i.$ Using Proposition \ref{prop:prob_of_hard_crossing} together with the obvious bound
$\mu(Q_i \text{ is good})\ge 1-(2^d-1)(1-\pi_\ell)$, the proof of the latter is now identical to the one given in
\cite{MT}*{Theorem 2.6}. 
\end{proof}

\subsection{Constructing the canonical path on the coarse-grained lattice}
\label{sec:paths}
In this section we will construct a set of $T$-step moves -- sequences
of at most $T\in \bbN$ legal moves for the KA dynamics (\ie legal exchanges or resampling of boundary sites) that could be chained together
in order to flip the state of an arbitrary point $x\in\bbZ^{d}$.
The construction of the move is quite cumbersome, so we will only
give here the required definitions and the statement of the result.
For details see \cite{A.Shapira}*{Chapter 5} and the supplementary file to the arXiv version of this paper.

The next definition describes how to move from one configuration to
the other using only legal KA exchanges. It will provide a way to construct, for a given initial configuration, a certain path in configuration space. We emphasise that, unlike the paths introduced earlier, this \emph{is not a geometric path} in $\Lambda_\ell$, but a path in the much larger configuration space $\Omega_\L$.
\begin{definition}[$T$-step move] Fix an integer $k\leq d$.
Given a finite connected subset $V$ of $\L$ and $\cM\subset \Omega$, a $T$-step move for KA-$k$f dynamics $M=(M_0,\dots,M_T)$ taking place in $V$ and
with domain $\text{Dom(M)}=\cM$ is a function from $\cM$ to
$\Omega^{T+1}$ such that  the sequence
$M(\eta)=\left(M_{0}(\eta),\dots,M_{T}(\eta)\right), \eta\in \cM,$
satisfies:
\begin{enumerate}[(i)]
\item $M_{0}(\eta)=\eta$,
\item for any $t\in [T]$, the configurations $M_{t-1}(\eta)$
and $M_{t}(\eta)$ are either identical or linked by a legal
move contained in $V$, {\ie either a legal KA-$k$f exchange among sites $x,y\in V$ or a resampling at a boundary site $z\in\partial V$}.
\end{enumerate}
\end{definition}

\begin{definition}[Information loss and energy barrier]
Given a $T$-step move $M$ its \emph{information loss} ${\text{Loss}_{t}(M)}$ at time $t\in [T]$ is defined as
\[
2^{\text{Loss}_{t}(M)}=\sup_{\eta'\in\text{Dom}(M)}\#\left\{
  \eta\in \text{Dom(M)}\,|\,M_{t}(\eta)=M_{t}(\eta^{\prime}),\,M_{t+1}(\eta)=M_{t+1}(\eta^{\prime})\right\} .
\]
In other words, knowing $M_{t}(\eta)$ and $M_{t+1}(\eta)$, we are guaranteed
that $\eta$ is one of at most $2^{\text{Loss}_t(M)}$ configurations.
We also set
$\text{Loss}(M):=\sup_{t}\text{Loss}_{t}(M)$.
The energy
barrier of $M$ is defined as
\[
E(M)=\sup_{\eta\in\text{Dom}(M)}\sup_{t\in\left[T\right]}\left |\#\left\{ \text{empty sites in }M_{t}(\eta)\right\} -\#\left\{ \text{empty sites in }\eta\right\} \right |.
\]
\end{definition}

The main result is the following proposition that guarantees the existence of a $T$-step move with a bounded information loss and energy barrier that allows to flip the configuration at $x$ (namely to go from $\eta\to\eta^x$) provided $\eta$ has a certain up-right good path.  Recall that $Q_i=
i+ \{0,1\}^d \setminus \{0\}^d\subset \bbZ^d_\ell$ and that $N=L/\ell$.
\begin{proposition}
\label{prop:T-step-move}Fix an integer $k\leq d$. Fix $i\in\Lambda_\ell$ and $x\in B_i$. If $i+\vec e_1 \in
\Lambda_\ell$ fix also an up-right path $\gamma$ connecting $i+\vec e_1$ to $\partial \Lambda_\ell$.
Then there exists a $T$-step move $\text{M}$ with
$$\text{Dom(M)}=\left\{ \eta\,|\,\gamma\text{ is good and all
      $j\in Q_i\cap \Lambda_\ell$} \text{ are good}\right\},$$
taking place in 
$\cup_{j\in \gamma}B_{j}\cup (Q_i\cap \Lambda_\ell)$ and
such that, for
all $\eta\in \text{Dom(M)}$ and all $t\in [T],$ $M_t (\eta)\in
\text{Dom(M)}$ and $M_T (\eta)$ is the configuration $\eta$
flipped at $x$. Moreover, $E\left(\text{M}\right)\le C \ell^{k-1},$
and for all $j\in \L_\ell$

\begin{align*}
&\text{Loss(M)}\le C\log_2 (\ell)\ell,\quad &T\le C N \ell^\lambda,
                                                  \quad
                                                  &\left|\mathcal{T}_{\text{M}}^{(j)}\right|
                                                  \le C \ell^\lambda
                                                  \, \qquad &\text{ for
                                                    } k=2 \\
&\text{Loss(M)}\le C\ell^{d},
\quad &T\le C N 2^{\ell^d}, \quad &\left|\mathcal{T}_{\text{M}}^{(j)}\right| \le C 2^{\ell^d} \,\qquad &\text{ for } k\ge 3
\end{align*}
    where $\mathcal{T}_{\text{M}}^{(j)}$ is
  the set of indices $t\in [T]$ such that for some $\eta\in \text{Dom(M)}$
  the configurations
  $\left(\text{M}_{t}(\eta),\text{M}_{t+1}\eta\right)$ are linked
  together by a legal KA-transition inside $B_j$. The constants $C,\lambda$ may depend only on $k$ and $d$.
\end{proposition}
In order to flip the state of a site $x$ we must perform a legal KA
exchange touching $x$ which in turn requires having enough empty sites in the
vicinity of $x$. Patches of empty sites (e.g. a super-good box \ie  a
good box with an
empty row for $d=k=2$) are certainly present inside the percolation structure of the good
boxes. However,  typically they will be
quite far from $x$ because $q\ll 1.$ The main idea behind the proof of
the proposition is to prove that such super-good boxes can
be moved at will inside the good percolation network and brought near
$x$ by concatenating 
suitable ``elementary'' $T$-steps moves. This concatenation will
form the sought global $T$-step move $M$. 

Unfortunately the general
construction of the elementary moves is a bit cumbersome and
technical. For the interested reader  we refer to \cite{A.Shapira}*{Chapter 5} and
to the
supplementary file attached to the arXiv version of this paper. Still,
we will present a sketch of the construction for the particular case $k=d=2$ that will give a flavour of the type of arguments used there.
\subsubsection{Sketch of the proof of Proposition 
\ref{prop:T-step-move} for $k=2,d=2$.}
\label{sec:figures}
We will first introduce the notion of almost good -- a box is said to be \emph{almost good} if it contains at least one empty site in every line and every column. Recall that a good box is a box that remains almost good even after filling one of its sites.

\begin{claim}[Exchanging rows]
\label{claim:exchangerows}
Fix $y\in \bbZ^d$, and consider the configurations in which the row $y
+ [\ell]\times\{0\}$ is empty, and the row above it contains at least
one empty site. Then there exists a $T$-step move $M$ whose domain
consists of these configurations, and in the final state $M_T(\eta)$ the rows $y +
[\ell]\times\{0\}$ and $y + [\ell]\times\{1\}$ are
exchanged. Moreover, $\text{\rm Loss}(M)=O(\log_2 \ell)$ and $T=O(\ell)$. See Figure \ref{fig:framedexchange}.  
\end{claim}
\begin{figure}[!ht]
\begin{tikzpicture}[scale=0.4, every node/.style={scale=0.7}]
    	\def \xx{0}
    	\def \yy{0}
		
		\draw[step=1, gray, very thin] (\xx,\yy) grid (\xx+4,\yy+2);
		\foreach \x in {0,...,3}{
			\draw (\xx + \x+0.5,\yy+0.5) node[black] {$0$};
		}
		\draw (\xx + 0.5,\yy+1.5) node[black] {$0$};
		
		\draw[->]  (\xx+5,\yy+1) to (\xx+7,\yy+1);
		
		\def \xx{8}
		\draw[step=1, gray, very thin] (\xx,\yy) grid (\xx+4,\yy+2);
		\foreach \x in {0,...,3}{
			\ifnum \x=1
				\draw (\xx + \x+0.5,\yy+1.5) node[black] {$0$};
			\else
				\draw (\xx + \x+0.5,\yy+0.5) node[black] {$0$};
			\fi
		}
		\draw (\xx + 0.5,\yy+1.5) node[black] {$0$};
		
		\draw[->]  (\xx+5,\yy+1) to (\xx+7,\yy+1);
		
		\def \xx{16}
		\draw[step=1, gray, very thin] (\xx,\yy) grid (\xx+4,\yy+2);
		\foreach \x in {1,...,3}{
			\draw (\xx + \x+0.5,\yy+0.5) node[black] {$0$};
		}
		\foreach \x in {0,...,1}{
			\draw (\xx + \x+0.5,\yy+1.5) node[black] {$0$};
		}
		
		\draw[->]  (\xx+5,\yy+1) to (\xx+7,\yy+1);
		
		\def \xx{0}
    	\def \yy{-4}
		\draw[step=1, gray, very thin] (\xx,\yy) grid (\xx+4,\yy+2);
		\foreach \x in {1,...,3}{
			\ifnum \x=2
				\draw (\xx + \x+0.5,\yy+1.5) node[black] {$0$};
			\else
				\draw (\xx + \x+0.5,\yy+0.5) node[black] {$0$};
			\fi
		}
		\foreach \x in {0,...,1}{
			\draw (\xx + \x+0.5,\yy+1.5) node[black] {$0$};
		}
		
		\draw[->]  (\xx+5,\yy+1) to (\xx+7,\yy+1);
		
		\def \xx{8}
		\draw[step=1, gray, very thin] (\xx,\yy) grid (\xx+4,\yy+2);
		\foreach \x in {2,...,3}{
			\draw (\xx + \x+0.5,\yy+0.5) node[black] {$0$};
		}
		\foreach \x in {0,...,2}{
			\draw (\xx + \x+0.5,\yy+1.5) node[black] {$0$};
		}
		
		\draw[->]  (\xx+5,\yy+1) to (\xx+7,\yy+1);
		
		\def \xx{16}
		\draw[step=1, gray, very thin] (\xx,\yy) grid (\xx+4,\yy+2);
		\foreach \x in {2,...,3}{
			\ifnum \x=3
				\draw (\xx + \x+0.5,\yy+1.5) node[black] {$0$};
			\else
				\draw (\xx + \x+0.5,\yy+0.5) node[black] {$0$};
			\fi
		}
		\foreach \x in {0,...,2}{
			\draw (\xx + \x+0.5,\yy+1.5) node[black] {$0$};
		}
		
		\draw[->]  (\xx+5,\yy+1) to (\xx+7,\yy+1);
		
		\def \xx{0}
		\def \yy{-8}
		\draw[step=1, gray, very thin] (\xx,\yy) grid (\xx+4,\yy+2);
		\draw (\xx +1.5,\yy+0.5) node[black] {$0$};
		\foreach \x in {0,...,3}{
			\draw (\xx + \x+0.5,\yy+1.5) node[black] {$0$};
		}
		
		\draw[->]  (\xx+5,\yy+1) to (\xx+7,\yy+1);
		
		\def \xx{8}
		\draw[step=1, gray, very thin] (\xx,\yy) grid (\xx+4,\yy+2);
		\draw (\xx +0.5,\yy+0.5) node[black] {$0$};
		\foreach \x in {0,...,3}{
			\draw (\xx + \x+0.5,\yy+1.5) node[black] {$0$};
		}
		
\end{tikzpicture}

\caption{\label{fig:framedexchange}This figure shows how to exchange an
empty row with a neighbouring  row containing at least one empty site (see Claim \ref{claim:exchangerows}). The loss comes from the fact that there are $\ell$ options for the position of the empty site in the upper row.}
\end{figure}
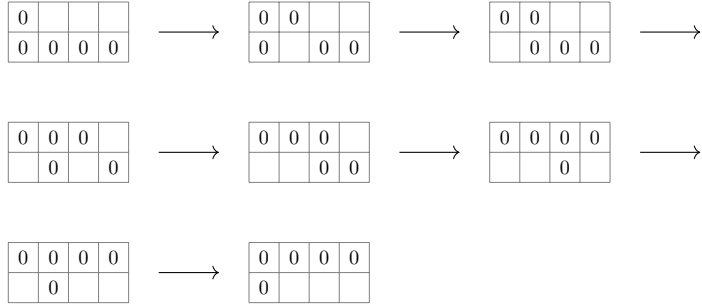

Note that even though the claim is formulated for exchanging rows, the same will hold for columns.

The path $\g$ in Proposition \ref{prop:T-step-move} is a general up-right path, but imagine for the moment that it is a straight path going right all the way to the boundary of $\L_\ell$.
 In this case, we can create an empty column on the boundary of $\Lambda$, and use Claim \ref{claim:exchangerows} to propagate it until it reaches the right side of the box $B_i$.

Assume further that all of the boxes in $Q_i$ are connected to the boundary by straight paths. The same construction as before will then allow us to empty all the sites in the outer up-right boundary of $B_i$. Figure \ref{fig:siteexchangemove} shows how in this state one can permute sites in $B_i$, so in particular we are able to change the occupation at $x$.

This, however, only allows us to exchange $x$ with another site in the same box; in order to flip its state without changing the other sites we must exchange it with a boundary site, which is connected to the reservoir (namely, it is being resampled from equilibrium, and in particular the number of particles is not conserved). In order to move the site to the other side of the empty column we use the following claim:

\begin{claim}[Moving a marked site] \label{claim:sitejump}
Fix $y\in \bbZ^d$ and some marked site $\star\in y + \{-1\}\times[\ell]$. Consider the configurations in which the column $y + \{0\}\times[\ell]$ is empty, and each of the columns $y + \{\pm 1\}\times[\ell]$ contains at least one empty site, not counting the site $\star$. Then there exists a $T$-step move $M$ whose domain
consists of these configurations, and in the final state $M_T(\eta)$
the sites $\star$ and $\star + (2,0)$ change their occupation
values. Moreover, $\text{\rm Loss}(M)=O(\log_2 \ell)$ and $T=O(\ell)$. See Figure \ref{fig:sitejump}.
\end{claim}
For this (very untypical) case, when the paths have no turns, the last claim will finish the construction -- after emptying the outer up-right boundary of $B_i$ and bringing the site $x$ to the rightmost column of $B_i$, apply consecutively Claim \ref{claim:exchangerows} and Claim \ref{claim:sitejump}. The energy barrier is at most $3\ell$ since we empty three rows/columns, and the time is $O(L \ell^2)$. The loss of information is also $O(\ell)$, since at each step we only need to reconstruct the three empty row/columns. If we are in the course of a move described in Claim \ref{claim:exchangerows} we must also pay its loss, but this gives a lower order contribution.

When the paths turn, however, we cannot propagate the empty line like before, and a more complicated mechanism is required. The first step consists in framing a frameamble box (recall Definition \ref{def:frameable}). When $d=k=2$, it means the following:

\begin{claim}[Framing a box]\label{claim:framing}
Fix a box, and consider the configurations for which the box is almost good, and, in addition, its bottom row is empty.
Then there exists a $T$-step move $M$ whose domain
consists of these configurations, and in the final state $M_T(\eta)$ the left column is also empty. Moreover, $\text{\rm Loss}(M)=O(\ell \log_2 \ell)$ and $T=O(\ell^2)$. See Figure \ref{fig:frame}.
\end{claim}

Once a box is framed, we can use again Figure \ref{fig:siteexchangemove} in order to construct a general permutation inside it, and in particular we are able to reflect the configuration.
\begin{claim}[Reflecting a framed configuration]\label{claim:reflection}
Fix a box, and consider the configurations for which the box is framed, \ie, its bottom row and left column are empty.
Then there exists a $T$-step move $M$ whose domain
consists of these configurations, and in the final state $M_T(\eta)$ the configuration inside the box is reflected along the axis connecting its bottom left corner with its up right corner. Moreover, $\text{\rm Loss}(M)=0$ and $T=O(\ell^4)$.
\end{claim}

\begin{figure}[ht]
\begin{tikzpicture}[scale=0.4, every node/.style={scale=0.7}]
    \def \x{0}
    \def \y{0}	
	\draw[step=1, gray, very thin, shift={(\x,\y)}] (0,0) grid +(4,4);

	\foreach \xx in {0,1,2,3}{
		\draw[shift={(\x,\y)}] (\xx+0.5,0.5) node[black] {$0$};
	}

	\draw[shift={(\x,\y)}] (1.5,1.5) node[black] {$0$};
	\draw[shift={(\x,\y)}] (2.5,3.5) node[black] {$0$};
	\draw[shift={(\x,\y)}] (3.5,2.5) node[black] {$0$};
		
	\draw[->,shift={(\x,\y)}]  (5,2) to (7,2);
	
	\def \x{8}
    \def \y{0}
    \draw[step=1, gray, very thin, shift={(\x,\y)}] (0,0) grid +(4,4);

	\foreach \xx in {0,1,2,3}{
		\draw[shift={(\x,\y)}] (\xx+0.5,0.5) node[black] {$0$};
	}

	\draw[shift={(\x,\y)}] (0.5,1.5) node[black] {$0$};
	\draw[shift={(\x,\y)}] (2.5,3.5) node[black] {$0$};
	\draw[shift={(\x,\y)}] (3.5,2.5) node[black] {$0$};
	
	\draw[->,shift={(\x,\y)}]  (5,2.5) to (7,2.5);
	
	\def \x{16}
    \def \y{0}
    \draw[step=1, gray, very thin, shift={(\x,\y)}] (0,0) grid +(4,4);

	\foreach \xx in {0,1,2,3}{
		\draw[shift={(\x,\y)}] (\xx+0.5,1.5) node[black] {$0$};
	}

	\draw[shift={(\x,\y)}] (0.5,0.5) node[black] {$0$};
	\draw[shift={(\x,\y)}] (2.5,3.5) node[black] {$0$};
	\draw[shift={(\x,\y)}] (3.5,2.5) node[black] {$0$};
		
	\draw[->,shift={(\x,\y)}]  (5,2) to (7,2);
	
	\def \x{2}
    \def \y{-5}
    \draw[step=1, gray, very thin, shift={(\x,\y)}] (0,0) grid +(4,4);

	\foreach \xx in {0,1,2,3}{
		\draw[shift={(\x,\y)}] (\xx+0.5,1.5) node[black] {$0$};
	}

	\draw[shift={(\x,\y)}] (0.5,0.5) node[black] {$0$};
	\draw[shift={(\x,\y)}] (2.5,3.5) node[black] {$0$};
	\draw[shift={(\x,\y)}] (0.5,2.5) node[black] {$0$};
		
	\draw[->,shift={(\x,\y)}]  (5,2) to (7,2);
	
	\def \x{10}
    \def \y{-5}
    \draw[step=1, gray, very thin, shift={(\x,\y)}] (0,0) grid +(4,4);

	\foreach \xx in {0,1,2,3}{
		\draw[shift={(\x,\y)}] (\xx+0.5,3.5) node[black] {$0$};
	}

	\draw[shift={(\x,\y)}] (0.5,0.5) node[black] {$0$};
	\draw[shift={(\x,\y)}] (0.5,1.5) node[black] {$0$};
	\draw[shift={(\x,\y)}] (0.5,2.5) node[black] {$0$};
	
	\draw[->,shift={(\x,\y)}]  (5,2) to (7,2);
	
	\def \x{18}
    \def \y{-5}
    \draw[step=1, gray, very thin, shift={(\x,\y)}] (0,0) grid +(4,4);

	\foreach \yy in {0,1,2,3}{
		\draw[shift={(\x,\y)}] (0.5,\yy+0.5) node[black] {$0$};
	}

	\draw[shift={(\x,\y)}] (2.5,0.5) node[black] {$0$};
	\draw[shift={(\x,\y)}] (1.5,0.5) node[black] {$0$};
	\draw[shift={(\x,\y)}] (3.5,0.5) node[black] {$0$};
		
\end{tikzpicture}

\caption{\label{fig:frame}Framing an almost good box. See Claim \ref{claim:framing}}
\end{figure}
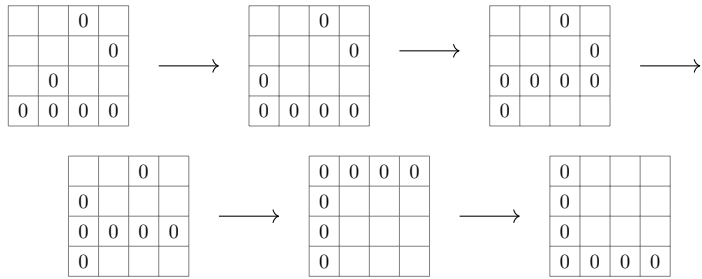

Now we are allowed, when reaching a turn of the path, to frame the box, reflect the configuration, "unframe" the reflected box, and continue propagating the marked site to finish the construction as in the corner-less case. The dominating contribution to the total loss is coming from the framing move, $O(\ell \log_2(\ell))$, the energy barrier remains the one coming from the creation of empty sites on the boundary, $O(\ell)$, and the time is $O(L\,\ell^4)$.

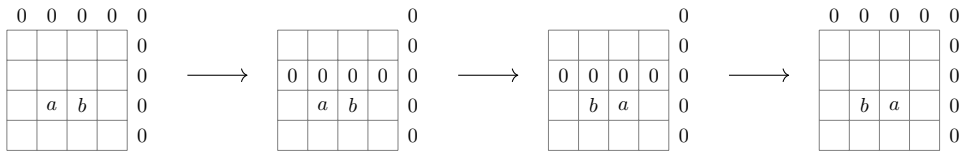
\begin{figure}[!ht]
\begin{tikzpicture}[scale=0.4, every node/.style={scale=0.7}]
	
	\def\xx{0}
	\def\yy{0}
	\draw[step=1, gray, very thin] (\xx,\yy) grid (\xx+4,\yy+4);
	\foreach \y in {0,...,4}{
			\draw (\xx+4.5,\yy+\y+0.5) node[black] {$0$};
		}
	\foreach \x in {1,...,4}{
			\draw (\xx+\x-0.5,\yy+4.5) node[black] {$0$};
		}
	\draw (\xx+1.5,\yy+1.5) node[black] {$a$};
	\draw (\xx+2.5,\yy+1.5) node[black] {$b$};
	
	\draw[->]  (\xx+6,\yy+2.5) to (\xx+8,\yy+2.5);
	
	\def\xx{9}
	\draw[step=1, gray, very thin] (\xx,\yy) grid (\xx+4,\yy+4);
	\foreach \y in {0,...,4}{
			\draw (\xx+4.5,\yy+\y+0.5) node[black] {$0$};
		}
	\foreach \x in {1,...,4}{
			\draw (\xx+\x-0.5,\yy+2.5) node[black] {$0$};
		}
	\draw (\xx+1.5,\yy+1.5) node[black] {$a$};
	\draw (\xx+2.5,\yy+1.5) node[black] {$b$};
	
	\draw[->]  (\xx+6,\yy+2.5) to (\xx+8,\yy+2.5);

	\def\xx{18}
	\draw[step=1, gray, very thin] (\xx,\yy) grid (\xx+4,\yy+4);
	\foreach \y in {0,...,4}{
			\draw (\xx+4.5,\yy+\y+0.5) node[black] {$0$};
		}
	\foreach \x in {1,...,4}{
			\draw (\xx+\x-0.5,\yy+2.5) node[black] {$0$};
		}
	\draw (\xx+1.5,\yy+1.5) node[black] {$b$};
	\draw (\xx+2.5,\yy+1.5) node[black] {$a$};
	
	\draw[->]  (\xx+6,\yy+2.5) to (\xx+8,\yy+2.5);
	
	\def\xx{27}
	\draw[step=1, gray, very thin] (\xx,\yy) grid (\xx+4,\yy+4);
	\foreach \y in {0,...,4}{
			\draw (\xx+4.5,\yy+\y+0.5) node[black] {$0$};
		}
	\foreach \x in {1,...,4}{
			\draw (\xx+\x-0.5,\yy+4.5) node[black] {$0$};
		}
	\draw (\xx+1.5,\yy+1.5) node[black] {$b$};
	\draw (\xx+2.5,\yy+1.5) node[black] {$a$};

\end{tikzpicture}

\caption{\label{fig:siteexchangemove}This figure shows how to exchange two
arbitrary neighbouring sites of a box if the external top row and right column are  empty. By a concatenation of such moves it is possible to exchange any two internal sites.}
\end{figure}

\begin{figure}[!ht]
\begin{tikzpicture}[scale=0.4, every node/.style={scale=0.7}]
    \def \x{0}
    \def \y{0}	
	\draw[step=1, gray, very thin, shift={(\x,\y)}] (0,0) grid +(3,4);

	\foreach \yy in {0,1,2,3}{
		\draw[shift={(\x,\y)}] (1.5,\yy+0.5) node[black] {$0$};
	}
	
	\draw[shift={(\x,\y)}] (0.5,1.5) node[black] {$0$};
	\draw[shift={(\x,\y)}] (2.5,3.5) node[black] {$0$};
	\draw[shift={(\x,\y)}] (0.5,2.5) node[black] {$\star$};
	\draw[shift={(\x,\y)}] (2.5,2.5) node[black] {$a$};
		
	\draw[->,shift={(\x,\y)}]  (4,2) to (5.5,2);
	
	\def \x{6.5}
    \def \y{0}
    \draw[step=1, gray, very thin, shift={(\x,\y)}] (0,0) grid +(3,4);

	\foreach \yy in {0,1,2,3}{
		\draw[shift={(\x,\y)}] (1.5,\yy+0.5) node[black] {$0$};
	}
	
	\draw[shift={(\x,\y)}] (0.5,0.5) node[black] {$0$};
	\draw[shift={(\x,\y)}] (2.5,0.5) node[black] {$0$};
	\draw[shift={(\x,\y)}] (0.5,2.5) node[black] {$\star$};
	\draw[shift={(\x,\y)}] (2.5,3.5) node[black] {$a$};
		
	\draw[->,shift={(\x,\y)}]  (4,2) to (5.5,2);
	
	\def \x{13}
    \def \y{0}
    \draw[step=1, gray, very thin, shift={(\x,\y)}] (0,0) grid +(3,4);

	\foreach \yy in {0,1,2,3}{
		\draw[shift={(\x,\y)}] (0.5,\yy+0.5) node[black] {$0$};
	}
	
	\draw[shift={(\x,\y)}] (1.5,0.5) node[black] {$0$};
	\draw[shift={(\x,\y)}] (2.5,0.5) node[black] {$0$};
	\draw[shift={(\x,\y)}] (1.5,2.5) node[black] {$\star$};
	\draw[shift={(\x,\y)}] (2.5,3.5) node[black] {$a$};
		
	\draw[->,shift={(\x,\y)}]  (4,2) to (5.5,2);
	
	\def \x{19.5}
    \def \y{0}
    \draw[step=1, gray, very thin, shift={(\x,\y)}] (0,0) grid +(3,4);

	\foreach \yy in {0,1,2,3}{
		\draw[shift={(\x,\y)}] (0.5,\yy+0.5) node[black] {$0$};
	}
	
	\draw[shift={(\x,\y)}] (1.5,0.5) node[black] {$0$};
	\draw[shift={(\x,\y)}] (2.5,0.5) node[black] {$0$};
	\draw[shift={(\x,\y)}] (2.5,3.5) node[black] {$\star$};
	\draw[shift={(\x,\y)}] (1.5,2.5) node[black] {$a$};
		
	\draw[->,shift={(\x,\y)}]  (4,2) to (5.5,2);
	
	\def \x{26}
    \def \y{0}
    \draw[step=1, gray, very thin, shift={(\x,\y)}] (0,0) grid +(3,4);

	\foreach \yy in {0,1,2,3}{
		\draw[shift={(\x,\y)}] (1.5,\yy+0.5) node[black] {$0$};
	}
	
	\draw[shift={(\x,\y)}] (2.5,0.5) node[black] {$0$};
	\draw[shift={(\x,\y)}] (0.5,0.5) node[black] {$0$};
	\draw[shift={(\x,\y)}] (2.5,3.5) node[black] {$\star$};
	\draw[shift={(\x,\y)}] (0.5,2.5) node[black] {$a$};
		
	\draw[->,shift={(\x,\y)}]  (4,2) to (5.5,2);
	
	\def \x{32.5}
    \def \y{0}
    \draw[step=1, gray, very thin, shift={(\x,\y)}] (0,0) grid +(3,4);

	\foreach \yy in {0,1,2,3}{
		\draw[shift={(\x,\y)}] (1.5,\yy+0.5) node[black] {$0$};
	}
	
	\draw[shift={(\x,\y)}] (2.5,3.5) node[black] {$0$};
	\draw[shift={(\x,\y)}] (0.5,1.5) node[black] {$0$};
	\draw[shift={(\x,\y)}] (2.5,2.5) node[black] {$\star$};
	\draw[shift={(\x,\y)}] (0.5,2.5) node[black] {$a$};
		
\end{tikzpicture}
\caption{\label{fig:sitejump} This figure shows how to make a marked site jump beyond an empty column. See Claim \ref{claim:sitejump}.}
\end{figure}
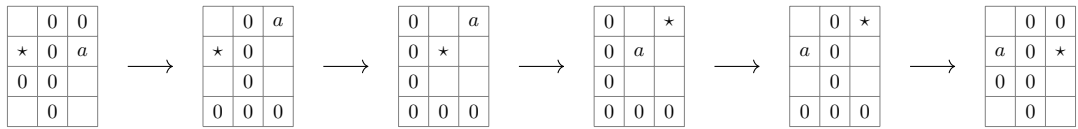

\subsection{From the long range Poincar\'e inequality to the Kob-Andersen dynamics}
\label{sec:longKA}
In this section we bound from above the Dirichlet form with the long
range constraints appearing in the r.h.s. of \eqref{eq:fabio1} with
that of the KA model in $\L$ \eqref{eq:Diri}. Given $i\in \L_\ell$ our aim is to bound the quantity $\mu(\hat c_i\var_{B_i}(f)$ appearing in the r.h.s. of \eqref{eq:fabio1} using the $T$-step moves that have been constructed in the previous section. In order to do that, it is convenient to first condition on the environment of the coarse-grained variables  
$\{{\mathds 1}_{\{B_j \text{ is good}\}}\}_{j\in \L_\ell,j\neq
  i}$. The main advantage of the above conditioning is that the good family for the vertex $i+\vec e_1$, whose existence is guaranteed by the long range constraint $\hat c_i$, become deterministic. We will thus work first in a fixed realisation of the coarse-grained variables satisfying $\hat c_i=1$ and only at the end we will take an average and we will sum over $i$. The main technical step of the above program is as follows.     

Given
$i\in \L_\ell$ let $\g$ be an up-right focused path $\gamma$ of length
$2N$ starting at $i+\vec{e}_1$ and let $G_{i,\g}$ be the event that $\gamma$ is
    good and all $j\in Q_i\cap \L_\ell$ are good.
Let also $V_{i,\g}:={B_i} \cup_{j\in \g \cup Q_i}B_j$ and let
$\cF_i$ be the $\s$-algebra generated by the random variables
$\{{\mathds 1}_{\{B_j \text{ is good}\}}\}_{j\in \L_\ell, j\neq
  i}$.  Notice that $G_{i,\g}$ is measurable w.r.t. $\cF_i$. Finally write
\[
  \cD_{i,\g}(f\tc \cF_i):=\sumtwo{x,y\in
    V_{i,\g}\cap \L}{\|x-y\|_1=1}\mu\big(c_{xy}\big(\nabla_{xy}f\big)^2\tc
  \cF_i\big) + \sum_{x\in V_{i,\g}\cap \partial \Lambda}\mu\big(\var_x(f)\tc \cF_i\big),
  \]
  where $\var_x(f)$, as before, is the variance conditioned on the occupation of the sites in $\L \setminus \{x\}.$
Clearly the average w.r.t. $\mu$ of $\cD_{i,\g}(f\tc \cF_i)$
represents the contribution coming from the set $V_{i,\g}\cap \L$ to the
total Dirichlet form $\cD(f)$.

For simplicity in the sequel we shall write $C(\ell,q)$ for any positive
function
 such that, as $\ell\uparrow +\infty, q\downarrow
0,$
\begin{align}\label{eq:gl}
  C(\ell,q)&=e^{O(\ell |\log(q)|+
  \ell \log(\ell))} & \text{for } k=2, \\
  C(\ell,q)&=e^{O(\ell^{k-1}|\log(q)|
  +\ell^d \log(\ell))} & \text{for } k\ge 3.
\end{align}
Of course the constant in the $O(\cdot)$
notation may change from line to line. 
\begin{lemma}
  \label{lem:path}
  On the event $G_{i,\g}$ 
  \[
\mu\big(\var_{B_i}(f)\tc \cF_i\big)\le O(N)C(\ell,q)\,\cD_{i,\g}(f\tc \cF_i)\quad \forall f:\O_\L\mapsto \bbR.
  \]
\end{lemma}
\begin{proof}
Assume ${\mathds 1}_{G_{i,\g}}=1$. Since the marginal of $\mu(\cdot\tc
\cF_i)$ on $\{0,1\}^{B_i}$ is a
product measure we have immediately that
\[
 \mu\big(\var_{B_i}(f)\tc \cF_i\big)\le \sum_{x\in
   B_i}\mu\big(\var_{x}(f)\tc \cF_i\big),
 \]
and it is sufficient to prove that
 \[
   \max_{x\in B_i}\mu\big(\var_{x}(f) \tc \cF_i\big)\le O(N) C(\ell,q)\,\cD_{{i,\g}}(f \tc \cF_i).
   \]
 Given $x\in B_i,$ Proposition
\ref{prop:T-step-move} and the assumption ${\mathds 1}_{G_{i,\g}}=1$
imply that there exists a $T$-step move $M$ with
$\text{Dom(M)}=G_{i,\g},$ 
taking place in 
$V_{i,\g}\cap \L$ and
such that for
all $\eta\in \text{Dom(M)}$ $M_T\eta$ is the configuration $\eta$
flipped at $x$. Notice that $M$ does not change the variables $\{{\mathds 1}_{\{B_j \text{ is good}\}}\}_{j\in \L_\ell, j\neq
  i}$. Hence, on the event $G_{i,\g}$,
\begin{align}
\var_x(f) =pq\,\big(f(\eta)-f(\eta^{x})\big)^{2}&\le \Big(\sum_{t=0}^{T-1}\,\big(f({M}_{t}(\eta))-f({M}_{t+1}(\eta))\big)\,\Big)^{2}\nonumber \\
 & \le T \sum_{t=0}^{T-1}\,\big(f({M}_{t}(\eta))-f({M}_{t+1}(\eta))\big)^{2}.
   \label{eq:fabio3}
\end{align}
In order to proceed it is convenient to introduce the following
notation.

A pair of configurations $e=(\eta,\eta')\in \O^2$ is called a
\emph{KA-edge} if $\eta\neq \eta'$ and $\eta'$ is obtained from $\eta$
by applying to $\eta$ either a legal exchange at some bond $b_e$ of
$\L$ or a spin flip at some site $z_e\in \partial \L$. If $b_e$ or
$z_e$ belong to a given $V\subset \L$ we say that \emph{the edge $e$ occurs in
$V$}. Given a KA-edge $e=(\eta,\eta')$
we write $\nabla_{e}f:=f(\eta')-f(\eta)$. Finally the collection of all
KA-edges in $\O^2$ is denoted $\O_{\rm KA}$.

By construction, if ${M}_{t+1}(\eta)\neq {M}_{t}(\eta)$ then
$e_t:=({M}_{t}(\eta) , {M}_{t+1}(\eta) )$ is a KA-edge and the
r.h.s. of \eqref{eq:fabio3} can be written as
\[
T\,\sum_{t=0}^{T-1}c_{e_t}\Big(\nabla_{e_t}f\Big)^{2},
\]
where $c_{e_t}$ is the kinetic constraint associated to the KA-edge
$e_t.$ Taking the expectation over $\eta$ w.r.t. $\mu(\cdot\tc \cF_i)$
yields
\begin{equation}
  \label{eq:fabio4}
\mu\Big(\var_x(f)\tc \cF_i\Big)\le
T \sum_{e\in \O_{\rm KA}}  \sum_{t={0}}^T \, 
\mu\big(c_e\big(\nabla_{e}f\big)^{2}\mathds{1}_{\{e=({M}_{t}(\eta),{M}_{t+1}(\eta))\}}\tc
\cF_i\big).
\end{equation}
Next we use the following chain of observations (recall Proposition
\ref{prop:T-step-move} and the relevant definitions therein).
\begin{enumerate}[(i)]
\item For any KA-edge $e$
and any $\eta$ such
that  $e=({M}_{t}(\eta),{M}_{t+1}(\eta))$ for some $t\le T$ it
holds that (for $q<1/2$)
\[
  \mu(\eta)  \le { q^{-E({M})}\mu({M}_{t}(\eta) )}.
  \]
\item Since the $T$-move M takes place in
  the set $V_{i,\g}\cap \L,$ in the r.h.s. of \eqref{eq:fabio4} we can replace
  $\sum_{e\in \O_{\rm KA}}$ by
  \[
    \sumtwo{e\in \O_{\rm KA}}{e \text{ occurs in
        $V_{i,\g}\cap \L$}}.
    \]
\item Given a KA-edge $e$ occurring in some $B_j\subset V_{i,\g}\cap \L,$ 
  \[
    \sum_{\eta\in \O}\sum_{t=1}^T
    \mathds{1}_{\{e=({M}_{t}(\eta),{M}_{t+1}(\eta))\}}
    \le 2^{\text{Loss(M)}}  |\mathcal{T}_{{M}}^{(j)}|.
      \]
\end{enumerate}
Using the above remarks, on the event $G_{i,\g},$
\begin{gather*}
\mu\big(\var_x(f)\tc \cF_i\big)\le T\, 2^{\text{Loss(M)}}  |\mathcal{T}_{\text{M}}^{(j)}| q^{E(\text{M})} \,
\sumtwo{e=(\eta,\eta')\in \O_{\rm KA}}{e \text{ occurs in
    $V_{i,\g}\cap \L$}}\mu(\eta\tc \cF_i)c_e(\eta)\big(\nabla_{e}f\big)^{2}.
\end{gather*}
This expression, by Proposition \ref{prop:T-step-move}, satisfies the required bound.
 \end{proof}
We are now ready to state the main result of this section.
\begin{proposition}
  \label{prop:long-vs-KA}
Let
$\mathcal{D}^{\ell}(f)=\mu\big(\sum_{i\in\L_\ell}\hat{c}_{i}\var_{B_i}(f)\big)$
and let $\cD(f)$ be the Dirichlet form of the KA model.
Then
\[
  \mathcal{D}^{\ell}(f)\le O(N^2)C(\ell,q)\cD(f).
  \]
\end{proposition}
\begin{corollary}
  \label{cor:final proof}
Fix $2\le k\le d$ together with $q\in (0,1)$. Assume that it is
possible to choose the mesoscopic scale $\ell$ depending only on $k,d,q$
in such a
way that $\pi_\ell(d,k)\ge \pi_*,$ where $\pi_*$ is the constant
appearing in Proposition \ref{prop:longrange1}. Then
\[
  \var(f)\le O(N^2)C(\ell,q)\cD(f).
\]
Equivalently
\[
 \trel(q,L)\le O(L^2)C(\ell,q).
  \]
\end{corollary}
\begin{proof}[Proof of the Corollary]
The first part of the corollary follows at once from Propositions \ref{prop:longrange1} and \ref{prop:long-vs-KA}. The second
part is an immediate consequence of the first one and of the
variational characterisation of the relaxation time (see the beginning
of Section \ref{sec:core}).
\end{proof}
 \begin{proof}[ Proof of Proposition \ref{prop:long-vs-KA}]
Recall definition \eqref{eq:long_range_contstaint} of the long range
constraints $\hat c_i$ and let us consider one term $\mu\big(\hat c_i\var_{B_i}(f)\big)$ appearing in the definition of
$\mathcal{D}^{\ell}(f) $. Observe that $\hat c_i$ is measurable
w.r.t. the $\s$-algebra $\cF_i$. Conditionally on $\cF_i$ and assuming that $\hat
c_i=1$, let $\cG$ be a family of good paths for the vertex $i+\vec
e_1+\vec e_2\in \bbZ^d_\ell$. Clearly $\hat c_i=1$ implies that
$G_{i,\g}$ occurs for each path $\g\in \cG$. Hence, by applying Lemma
\ref{lem:path} to each path in $\cG$ we get 
\begin{gather}
\mu\big(\var_{B_i}(f)\tc \cF_i\big)
\le O(N)C(\ell,q)\,\frac{1}{|\cG|}\sum_{\gamma\in\cG}\cD_{i,\g} \nonumber
\\
=O(N)C(\ell,q)\Big[\,\sumtwo{x,y\in
  \L}{\|x-y\|_1=1}\mu\big(c_{xy}\big(\nabla_{xy}f\big)^2\tc \cF_i\big)\frac{1}{|\cG|}\sum_{\gamma\in\cG}\mathds{1}_{\{(x,y)\subset V_{i,\g}\}} \nonumber\\
+ \sum_{x\in \partial \Lambda}\mu\big(\var_x(f)\tc
\cF_i\big) \frac{1}{|\cG|}\sum_{\gamma\in\cG}\mathds{1}_{\{x\in V_{i,\g}\}}\,\Big].
 \label{eq:D_averaged_over_paths}
\end{gather}
For a given bond $(x,y)\subset \L$ (respectively $x\in \partial \L$)
let $j=j(x)$ be such that $B_j\ni x$ and let $\Pi_{j}$ denote the
$(\vec e_1,\vec e_2)$-plane in $\bbZ^d_\ell$ containing $j$. Since all
the paths forming the family $\cG$ belong to the plane
$\Pi_i,$ {and are focused} we immediately get that
\begin{align*}
\frac{1}{|\cG|}\sum_{\gamma\in\cG}\mathds{1}_{\{(x,y)\subset
  V_{i,\g}\}}&=\mathds{1}_{\{{j\in \Pi_{i}}\}} {\mathds{1}_{\{{j\in \mathcal R_{i}}\}}} \,\frac{1}{|\cG|}\sum_{\gamma\in\cG}\mathds{1}_{\{(x,y)\subset
  V_{i,\g}\}},\\
 \frac{1}{|\cG|}\sum_{\gamma\in\cG}\mathds{1}_{\{x\in \partial
  V_{i,\g}\}}&=\mathds{1}_{\{{j\in \Pi_{i}}\}} {\mathds{1}_{\{{j\in \mathcal R_{i}}\}}} \,\frac{1}{|\cG|}\sum_{\gamma\in\cG}\mathds{1}_{\{x\in \partial
  V_{i,\g}\}} 
  \end{align*}
  {{ where $ \mathcal R_{i}$ is the set of points at distance at most $\sqrt N$ from the set $\{k: k=i+s(\vec e_1+\vec e_2), s\in\mathbb N\}$}}.\\
Next, for $(x,y)\subset \L$ (respectively $x\in \partial \L$) such that
$\|i-j\|_{1}\le\sqrt{N}$ we bound
$\frac{1}{|\cG|}\sum_{\gamma\in\cG}\mathds{1}_{\{(x,y)\subset
  V_{i,\g}\}}$ (respectively 
$\frac{1}{|\cG|}\sum_{\gamma\in\cG}\mathds{1}_{\{x\in \partial
  V_{i,\g}\}}$) by one. 
If instead $\|i-j\|_{1}> \sqrt{N}$ then we use the fact that the paths of $\cG$ are almost edge-disjoint to bound
from above both sums by $2/|\cG|\le 2/\sqrt{N}$.

In conclusion, the first and
second term inside the square bracket
in the r.h.s. of \eqref{eq:D_averaged_over_paths} are  bounded from
above by
\[
\sumtwo{x,y\in
  \L}{\|x-y\|_1=1}\mu\big(c_{xy}\big(\nabla_{xy}f\big)^2\tc
\cF_i\big)\mathds{1}_{\{{j\in \Pi_{i}}\}}  {{\mathds{1}_{\{j\in \mathcal R_{i}\}}}}\big[\mathds{1}_{\{\|i-j(x)\|_1\le \sqrt{ N}\}}+ \frac{2}{\sqrt{N}}\mathds{1}_{\{\|i-j(x)\|_1> \sqrt{ N}\}}\big]
  \]
  and
  \[
\sum_{x\in \partial \Lambda}\mu\big(\var_x(f)\tc \cF_i\big)\mathds{1}_{\{{j\in \Pi_{i}}\}} {{\mathds{1}_{\{j\in \mathcal R_{i}\}}}} \big[\mathds{1}_{\{\|i-j(x)\|_1\le \sqrt{N}\}} + \frac{2}{\sqrt{N}}\mathds{1}_{\{\|i-j(x)\|_1> \sqrt{N}\}}\big]
\]
respectively. Clearly the same bounds hold for their average w.r.t. $\mu$.  

In order to conclude the proof it is enough to sum over $i$ the above
expressions and use the fact that, uniformly in $x\in \L$, 
\[
\sum_{i\in \L_\ell}\mathds{1}_{\{{j\in \Pi_{i}}\}}
{\mathds{1}_{\{j\in \mathcal R_{i}\}}}
\big[\mathds{1}_{\{\|i-j(x)\|_1\le \sqrt{ N}\}}+ \frac{2}{\sqrt{N}}\mathds{1}_{\{\| i-j(x)\|_1> \sqrt{ N}\}}\big]\le O(N).
\]
  \end{proof}

\subsection{Completing the proof of the upper bound}
\label{sec:proof} Using Corollary \ref{cor:final proof}, the proof of
the upper bound  is complete if we can prove that
{for all $\pi^*<1$}, for any given $q\in (0,1)$ and $2\le k\le d$ it is possible to choose
$\ell=\ell(q,k,d)$ in such a way that 
\begin{itemize}\item[(i)]
the probability that any given $i\in \bbZ^d_\ell$ is $(d,k)$-good  satisfies 
 $\pi_\ell(d,k)\ge \pi_*$ ;\item[(ii)]
$
C(\ell,q)\le C(q)
$
as $q\to 0$, where $C(q)$ is as in \eqref{eq:main} and $C(\ell,q)$ { satisfies \eqref{eq:gl}}.\end{itemize}
Let us start by stating a key result on the probability of the set of frameable configurations
\begin{proposition}[Probability of frameable configurations \cite{TBF}]
\label{prop:fra}
Fix $q$ and let $\mathcal{F}_q(\ell,d,j)$ be the probability that the cube $C_{\ell}=[\ell]^d$
is $(d,j)$-frameable. Then there exists $C>0$ s.t. for $q\to 0$  
$$\mathcal{F}_q(\ell,d,j)\geq 1-Ce^{-\ell_q/\Xi_{d,j}}\,\,\,\,\,\,\,\,{\forall \ell_q \,\,\mbox{ s.t. }\,\,\, \Xi_{d,k}(q)=O(\ell_q)}$$
with
$$\Xi_{d,1}(q):=\left(\frac{1}{q}\right)^{1/d}$$
and
$$\Xi_{d,j}(q):=\exp_{(j-1)}\left(\frac{1}{q^{\frac{1}{d-j+1}}}\right)\,\,\,\,\,\,\,\,\,\,\forall j\in [2,d].$$
\end{proposition}
\begin{proof}
The case $j=1$ follows immediately from the definition of frameable configurations (see Definition \ref{def:frameable}). The cases $j\in [2,d]$ are proven in Section 2 of \cite{TBF}, see formula (34) \footnote{There is a misprint in formula (34) of \cite{TBF}: in the exponential a minus sign is missing} and (36), where the results are stated in terms of the parameter $s=j-1$. Actually, the definition of frameable in \cite{TBF} is more restrictive than our Definition \ref{def:frameable}. Indeed in \cite{TBF} the frame that should be emptiable is composed by all the faces of dimension $j-1$ containing one corner of  $\mathcal C_{\ell}$ (and not only those that contain the vertex  $(1,\dots,1)$). However, since we only need a lower bound on the probability of being frameable we can directly use the results of \cite{TBF} .
\end{proof}
Then, by using Proposition \ref{prop:fra} and the Definition \ref{def:good_box}, we get
that there exists $c>0$ s.t. by choosing 
$$\ell(q,k,d)= \exp_{(k-2)} \left(c/q^{\frac{1}{d-j+1}}\right) \,\,\,\forall k\in[3,d]$$
and
$$\ell(q,2,d)=|\log q|/q^{\frac{1}{d-1}}$$
we get 
$$\pi_{\ell}(d,k)\geq \left(1-C\exp^{-\ell/\Xi_{d-1,k-1}}\right)^{\ell d}$$
which  goes to $1$ as $q\to 0$, and thus implies 
that condition (i) is satisfied for all $q\in(0,1)$ (since $\pi_{\ell}(d,k)$ is non decreasing with $q$) . 
Finally, it is immediate to verify that the above choice of $\ell$ satisfies also condition (ii) for all $k\in[2,d]$.

\section{Proof of the lower bound in Theorem \ref{maintheorem:1}} \label{sec:lowerbound}
In this section we will prove the lower bound on the relaxation time
by finding suitable positive constants $c,\lambda$ depending on $d,k$ and a function $f$ such that   
\begin{equation}
  \label{eq:lwb1}
\var(f)\ge e^{\lambda m}L^2 \cD(f),
\end{equation}
where $m:=m(q)$ satisfies
\begin{align}
m(q) & =\begin{cases}
\big\lfloor cq^{-\frac{1}{d-1}} \big\rfloor & k=2\\
\big\lfloor \exp_{(k-2)}\left(cq^{-\frac{1}{d-k+1}}\right)\big\rfloor & k\ge3.
\end{cases}\label{eq:lowerboundscale}
\end{align}
Note that $m$ defined here, up to logarithmic corrections, 
describe a length scale similar to $\ell$ of the previous section. Roughly speaking, this is the scale at which large structures that contain many empty sites could propagate and influence their neighborhood.

\subsection{Bootstrap percolation}
In order to define $f$ we first need to introduce the \emph{$k$-neighbor
bootstrap percolation} (see e.g. \cite{Robsurvey} and references therein). Fix $V\subseteq\Lambda$, and consider a set
$A\subseteq \L$. The $k$-neighbor bootstrap percolation in $V$ starting
at $A$ is the deterministic growth process in discrete time, defined as 
\begin{align*}
A_{0} & =A\cap V,\\
A_{t+1} & =A_{t}\cup\left\{ x\in V:\left|\left\{ y\in A_{t}\text{ such that }y\sim x\right\} \right|\ge k\right\}, \
          t\in \bbN .
\end{align*}
That is, at each step the set $A_{t+1}$ is obtained by adding to
the set $A_{t}$ the sites that have at least $k$ neighbors already
in $A_{t}$. The set $\cup_{t\ge0}A_{t}$
will be denoted by $\left[A\right]^{V}$ and it forms a subgraph of
$\Lambda$ whose connected components will be referred to as
\emph{clusters}. Given $x\in V$ we shall write $\cC_x^V$ for either
the cluster of $[A]^V$ containing $x$ if $x\in [A]^V$ or for the set $\left\{
  x\right\} $ otherwise. Given $\eta\in \O_V$ we shall define the
bootstrap percolation process started from $\eta$ as the above process
with initial set $A=A_\eta:=\left\{ x\in\Lambda:\eta_{x}=0\right\}.$  



\subsection{Construction of the test function}

We start with a few geometric definitions:
\begin{definition}
Denote the box $x+\left[-m,m\right]^{d}$ by $B_{x}$. Its inner
boundary is denoted by $\partial B_{x}$. We say that an edge $y\sim z$
\emph{crosses} $\partial B_{x}$, and write $yz\in\bar \partial B_{x}$,
if one of its endpoints is in \textbf{$B_{x}$} and the other is not.
\end{definition}

\begin{definition}
Fix a configuration $\eta\in\Omega$ and a site $x\in\Lambda$, and
consider the cluster of $x$ in $\left[A_{\eta}\right]^{B_{x}}$.
We define $r_{x}(\eta)$ to be the maximal site in this cluster, according
to the lexicographic order.
\end{definition}
We are now ready to define the function $f.$ Let $g:\left[0,1\right]^{d}\rightarrow\bbR$
a positive smooth function supported in $\left[0.1,0.9\right]^{d}$.
Then 
\begin{equation}
f(\eta)=\sum_{x\in\Lambda}g\big(r_{x}(\eta)/L\big)\eta_{x}.\label{eq:testfunction}
\end{equation}
\begin{remark}
The above choice is inspired by the test function $\varphi=
\sum_{x\in\Lambda}g\left(x/L\right)\eta_{x}$
for the symmetric simple exclusion process, with $g$ related to the
lowest eigenfunction of the
discrete Laplacian in $\L$ (see
e.g. \cite{octopus}*{Section 4.1}). The only crucial difference between $f$ and
$\varphi$ is 
the choice of $r_x(\eta)$ instead of $x$ inside the 
slowly varying function $g(r_x(\eta)/L)$ as a proxy for the \emph{effective
  position} of the particle at $x$. Actually any other quantity
depending only the cluster $\cC_x$ (e.g. its center of mass) would
work as well. Evaluating $g$ at this effective position is the cause of the prefactor $e^{\l m}$
in front of the diffusive term $L^2$ in \eqref{eq:lwb1}. In fact, as
proved below, the cluster $\cC_x$ is influenced by an exchange of
the KA dynamics in the box $B_x$ only if 
$\cC_x\cap \partial B_x\neq \emptyset.$ Since the latter event has
probability $e^{-O(m)}$ if the constant $c$ appearing in
\eqref{eq:lowerboundscale} is small enough, the
sought prefactor emerges.    
\end{remark}
We shall now bound separately the variance and Dirichlet form of
$f$. In the sequel, $c$ and $\lambda$ will denote generic positive
constants that may depend only on $k, d,$ and $g$.

\subsection{Bounding the variance}
\begin{proposition}
\label{prop:bigvar}For $q$ small enough and $L$ large enough, 
\[
\var(f)\ge c\,qL^{d}.
\]
\end{proposition}

\begin{proof}
Let $H=\left(2m+1\right)\bbZ^{d}$,
and for $\xi\in B_{0}$ let $H_{\xi}=\left(\xi+H\right)\cap\Lambda$.
Clearly $H_\xi\cap H_{\xi'}=\emptyset$ iff $\xi\neq \xi'$ and
$\cup_{\xi \in B_0}H_\xi=\L$. Note also that for any $x,x'\in H_{\xi},
x\neq x'$, $B_{x}\cap B_{x'}=\emptyset$. 

For $\xi\in B_{0}$ denote by $f_{\xi}(\eta)$ the part of the sum
in equation (\ref{eq:testfunction}) that corresponds to $H_{\xi}$:
\[
f_{\xi}(\eta)=\sum_{x\in H_{\xi}}g\left(r_{x}/L\right)\eta_{x}.
\]
By the previous observation this is a sum of independent variables so that 
\begin{align*}
\text{Var}(f_{\xi}) & =\sum_{x\in H_{\xi}}\text{Var}\left[g\left(\frac{1}{L}r_{x}(\eta)\right)\eta_{x}\right]=\sum_{x\in H_{\xi}}\text{Var}\left[(1+O(m/L))g\left(x/L\right)\eta_{x}\right]\\
 & =\sum_{x\in H_{\xi}}g(x/L)^{2}\,pq\,\left(1+O(m/L)\right)\ge\frac{1}{2}pq\sum_{x\in H_{\xi}}g(x/L)^{2}.
\end{align*}
The notation $O(\cdot)$ stands for a random variable deterministically
bounded by the expression inside the parentheses times a constant. 
Above we used $|r_x(\eta)|=O(m)$ to write 
\begin{equation}
  \label{eq:lwb2}
g(r_x(\eta)/L)=g(x/L)(1+
O(m/L)),  
\end{equation}
recalling that $g$ is smooth. 

Next, for $\xi\neq\xi'$, 
\begin{align*}
\text{Cov}\left(f_{\xi},f_{\xi'}\right) & =\sum_{x\in H_{\xi}}\sum_{x'\in H_{\xi'}}\text{Cov}\left(g\left(r_{x}/L\right)\eta_{x},g\left(r_{x'}/L\right)\eta_{x'}\right)\\
 & =\sum_{x\in H_{\xi}}\sum_{x\in H_{\xi'}}\mathds{1}_{\|x-x'\|\le2m+1}\text{Cov}\left(g\left(r_{x}/L\right)\eta_{x},g\left(r_{x'}/L\right)\eta_{x'}\right).
\end{align*}
Considering one of these terms, using equation \eqref{eq:lwb2} and $\cov(\eta_x,\eta_{x'})=0$, we
find that   
\[
\text{Cov}\left(g\left(r_{x}/L\right)\eta_{x},g\left(r_{x'}/L\right)\eta_{x'}\right)=O(m/L),
\]
yielding 
\[
\left|\text{Cov}\left(f_{\xi},f_{\xi'}\right)\right|=\sum_{x\in H_{\xi}}\sum_{x\in H_{\xi'}}\mathds{1}_{\|x-x'\|\le2m+1}\,O\left(m/L\right)\le\left|H_{\xi}\right|\left(4m+3\right)^{d}O(m/L).
\]
Putting everything together,
\begin{align*}
\text{Var}(f) & =\text{Var}\big(\sum_{\xi\in B_0}f_{\xi}\big)=\sum_{\xi}\text{Var}f_{\xi}+\sum_{\xi\neq\xi'}\text{Cov}\left(f_{\xi},f_{\xi'}\right)\\
 & \ge\sum_{\xi}\frac{1}{2}pq\sum_{x\in H_{\xi}}g(x/L)^{2}-\sum_{\xi\neq\xi'}\left|H_{\xi}\right|\,O\left(m^{d+1}/L\right)\\
 & =\frac{1}{2}pq\sum_{x}g\left(x/L\right)^{2}-O(m^{2d+1}L^{d-1})\ge\frac{1}{4}pq\,L^{d}\int g(s)^{2}\text{d}s
\end{align*}
for $L$ large enough.
\end{proof}

\subsection{Bounding the Dirichlet form}

In order to bound $\cD(f)$, we will use \cite{CerfManzo}*{Lemma 5.1}.
Plugging our choice of $m$ for small enough $c$ in their result yields the following
lemma\footnote{Note that in \cite{CerfManzo} the parameter $k$ is called $\ell$, and that the
length scale $m$ of equation (\ref{eq:lowerboundscale}) corresponds
(up to a constant) to $m_{-}$ of \cite{CerfManzo}.}:
\begin{lemma}
\label{lem:cerfmanzo}Consider the bootstrap percolation in the box
$B_{0}$ starting at $A_{\eta}$, where $\eta$ is a configuration
distributed according to $\mu$. Then the probability that the cluster
of the origin in $\left[A_{\eta}\right]^{B_{0}}$ contains a site
in $\partial B_{0}$ is bounded from above by $e^{-\lambda m}$. The same bound
holds replacing $B_{0}$ by the box $\left[-m,m+1\right]\times\left[-m,m\right]^{d}$
(or any of its rotations).
\end{lemma}

We will now make a few combinatorial observations.
\begin{observation}
\label{obs:bpmonotonicity}Fix a configuration $\eta$, and two sets
$U,V\subseteq\Lambda$. Then $\left[A_{\eta}\right]^{U}\subset\left[A_{\eta}\right]^{U\cup V}$.
\end{observation}

\begin{observation}
\label{obs:connectedtopivotalset}Fix a configuration $\eta$, two
sets $U\subseteq V\subseteq\Lambda$, and a site $x\in V\setminus U$.
Assume that $x\in\left[A_{\eta}\right]^{V}$, but $x\notin\left[A_{\eta}\right]^{V\setminus U}$.
Then $\cC_x^V\cap U\neq \emptyset$.
\end{observation}

\begin{observation}
\label{obs:3}
Fix a configuration $\eta$, a set $V\subseteq\Lambda$, and two sites
$y\sim z\in V$. Assume the constraint $c_{yz}$ is satisfied in $V$
(\ie, when fixing all sites outside $V$ to be occupied), and that $\eta_y \neq \eta_z$. Then both
$y$ and $z$ are contained in $\left[A_{\eta}\right]^{V}$. In
particular, $\cC^V_y=\cC^V_z.$
\end{observation}

\begin{observation}
\label{obs:bpafterkamove}Fix a configuration $\eta$, a set $V\subseteq\Lambda$,
and two sites $y\sim z\in V$. Assume the constraint $c_{yz}$ is
satisfied in $V$. Then $\left[A_{\eta}\right]^{V}=\left[A_{\eta^{yz}}\right]^{V}$.
\end{observation}

\begin{claim}
\label{claim:bpinshiftedbox}Fix a configuration $\eta$ and an edge
$y\sim z$, such that $c_{yz}=1$ and $\eta_y \neq \eta_z$. Assume that $\cC_y^{B_{y}\cup
  B_{z}}\cap \partial (B_y\cup B_z)=\emptyset$. Then $r_{y}(\eta)=r_{z}(\eta^{yz})$.
\end{claim}
\begin{proof}
Using observation \ref{obs:3} and the fact that $y\sim z$ we have that
$\cC_y^B=\cC_z^B$ for $B\in \{B_y,B_z, B_y\cup B_z\}$. Moreover, using
observation \ref{obs:bpafterkamove} these clusters are the same for
$\eta$ and $\eta^{yz}$.  
We will show that $\cC_y^B$ is the same for all three boxes, which
will imply the result. We start by showing that $\cC_y^{B_y}= \cC_y^{B_y \cup B_z}.$
By observation \ref{obs:bpmonotonicity} $\cC_y^{B_y}\subseteq \cC_y^{B_y \cup B_z}$. In the other direction,
let, by contradiction, $w\in
\cC_y^{B_z\cup B_y}\setminus [A_\eta]^{B_{y}}$. Then, setting $V=B_z \cup B_y$ and $U=(B_z\cup
B_y)\setminus B_{y}$,
observation \ref{obs:connectedtopivotalset} implies that
$\cC_y^{B_z\cup B_y} \cap U\neq \emptyset.$ Noticing that $U \subset \partial(B_y \cup B_z)$, this contradicts the
assumption of the claim. We conclude that $\cC_y^{B_z\cup B_y}=
\cC_y^{B_{y}}$. Similarly one proves that $\cC_y^{B_z\cup B_y}=
\cC_y^{B_{z}}$, and the result follows.
\end{proof}
\begin{claim}
\label{claim:contributionstodirichlet}Fix $x\in\Lambda$ and $y\sim z$.
Then:
\begin{enumerate}
\item If $yz\in\bar \partial B_{x}$ (so in particular $x\notin\left\{ y,z\right\} $),
\[
\mu(c_{yz}\left(\grad_{yz}\left[g\left(r_{x}/L\right)\eta_{x}\right]\right)^{2})\le e^{-\lambda m}\,m^{2}/L^{2}.
\]
\item If $x\in\left\{ y,z\right\} $,
\[
\mu\big(c_{zy}\big(\grad_{yz}\big[g\big(r_{y}/L\big)\eta_{y}+g\big(r_{z}/L\big)\eta_{z}\big]\big)^{2}\big)\le e^{-\lambda m}\,m^{2}/L^{2}.
\]
\item Otherwise, 
\[
c_{yz}\big(\grad_{yz}\big[g\big(r_{x}/L\big)\eta_{x}\big]\big)=0.
\]
\end{enumerate}
\end{claim}

\begin{proof}
Since $\grad_{xy}(\cdot)=0$ whenever $\eta_y = \eta_z$, we may assume throughout the proof of this claim that $\eta_y \neq \eta_z$, using freely observation \ref{obs:3} and Claim \ref{claim:bpinshiftedbox}.

For the first part, we note that $\eta_{x}$ does not change when
exchanging the sites $y$ and $z$, so the only contribution to $\grad_{yz}\left[g\left(r_{x}/L\right)\eta_{x}\right]$
comes from the change of $r_{x}$. Assume without loss of generality
that $y\in B_{x}$ and $z\notin B_x$. By observations \ref{obs:bpmonotonicity} and \ref{obs:connectedtopivotalset},
$\left[A_{\eta}\right]^{B_{x}}$ (and therefore $r_{x}$) cannot change
when changing $\eta_{y}$, unless $y\in \left[A_{\eta}\right]^{B_{x}} \cup \left[A_{\eta^{yz}}\right]^{B_{x}}$. Hence, by equation \eqref{eq:lwb2}, Lemma
\ref{lem:cerfmanzo} and defining $\cC_x$ to be $\cC_x^{B_x}(\eta)$ when $\eta_y = 0$ and $\cC_x^{B_x}(\eta^{yz})$ when $\eta_z=0$, 
\begin{align*}
\mu(c_{yz}\left(\grad_{yz}\left[g\left(r_{x}/L\right)\eta_{x}\right]\right)^{2}) & =\mu(c_{yz}\left(\grad_{yz}\left[g\left(r_{x}/L\right)\eta_{x}\right]\right)^{2}\mathds{1}_{y\in\cC_x})\\
 & \le O(m^{2}/L^{2})\mu(y\in\cC_x)\le e^{-\lambda  m}m^{2}/L^{2}.
\end{align*}

In order to prove the second part, note first that $y,z\in B_y\cap
B_z$ so that observation \ref{obs:bpafterkamove} together with
$c_{yz}(\eta)=1$ implies that
$\cC_y^{B_y}(\eta)=\cC_y^{B_y}(\eta^{yz})$ and in particular $r_y(\eta)=r_y(\eta^{yz})$. In the same manner $r_z(\eta)=r_z(\eta^{yz})$. Suppose now that
$r_{y}(\eta)=r_{z}(\eta^{yz})$, so also $r_{z}(\eta)=r_{y}(\eta^{yz})$. Then 
\[
g\left(r_{y}(\eta)/L\right)\eta_{y} + g\left(r_{z}(\eta)/L\right)\eta_{z}=
g\left(r_{z}(\eta^{yz})/L\right)\eta_{z}^{yz} + g\left(r_{y}(\eta^{yz})/L\right)\eta_{y}^{yz}.
\]
Therefore,  $c_{yz}\mathds{1}_{\{r_{y}(\eta)=r_{z}(\eta^{yz})\}}\grad_{yz}\left[g\left(r_{y}(\eta)/L\right)\eta_{y}+g\left(r_{z}(\eta)/L\right)\eta_{z}\right]=0$.
We are thus left with estimating 
\[
\mu\left(c_{xy}\mathds{1}_{\{r_{y}(\eta)\neq r_{z}(\eta^{yz})\}}\big(\grad_{yz}\left[g\left(r_{y}(\eta)/L\right)\eta_{y}+g\left(r_{z}(\eta)/L\right)\eta_{z}\right]\big)^{2}\right).
\]
Using \eqref{eq:lwb2} 
\[
|\grad_{yz}\left[g\left(r_{y}(\eta)/L\right)\eta_{y}+g\left(r_{z}(\eta)/L\right)\eta_{z}\right]|= O(m/L).
\]
Moreover, by Claim \ref{claim:bpinshiftedbox} and Lemma \ref{lem:cerfmanzo},
$\mu\left(c_{xy}\mathds{1}_{\{r_{y}(\eta)\neq r_{z}(\eta^{yz})\}}\right)\le e^{-\lambda m}$,
and this concludes the proof.

The third part is a direct consequence of Observation \ref{obs:bpafterkamove}.
\end{proof}
We are now ready to bound from above the Dirichlet form
$\mathcal{D}(f)$.
\begin{proposition}
\label{prop:smalldirichlet}For any small enough $c>0$ in
\eqref{eq:lowerboundscale} there exists $\l>0$ such that $\mathcal{D}(f)\le e^{-\lambda m}L^{d-2}$.
\end{proposition}
\begin{proof}
First, note that, since $g$ is supported in $\left[0.1,0.9\right]$,
the term $\sum_{x\in\partial\Lambda}\text{Var}_{x}(f)$ in $\cD(f)$ equals $0$.
Consider then one of the exchange terms $c_{yz}\left(\grad_{yz}f\right)^{2}$,
and split the sum over $x$ according to the different cases in Claim
\ref{claim:contributionstodirichlet}:
\begin{gather*}
c_{yz}\big(\grad_{yz}f\big)^{2}
=c_{yz}\big(\sum_{x\in\Lambda}\grad_{yz}\left[g\left(r_{x}(\eta)/L\right)\eta_{x}\right]\big)^{2}
\\ 
   =c_{yz}\big(\grad_{yz}\left[g\left(r_{y}(\eta)/L\right)\eta_{y}+g\left(r_{z}(\eta)/L\right)\eta_{z}\right]+\sum_{x:\,
   yz\in\bar\partial B_{x}}\grad_{yz}\left[g\left(r_{x}(\eta)/L\right)\eta_{x}\right]\big)^{2}\\
 \le
   cm^{d-1}c_{yz}\big(\left(\grad_{yz}\left[g\left(r_{y}(\eta)/L\right)\eta_{y}+g\left(r_{z}(\eta)/L\right)\eta_{z}\right]\right)^{2}+\sum_{x:\,
   yz\in\bar\partial B_{x}}\left(\grad_{yz}\left[g\left(r_{x}(\eta)/L\right)\eta_{x}\right]\right)^{2}\big),
\end{gather*}
where we have used the Cauchy-Schwarz inequality and the fact that
the sum $\sum_{x:\, yz\in\partial B_{x}}$ contains $2\left(2m+1\right)^{d-1}$
terms corresponding to the $\left(2m+1\right)^{d-1}$ possible
translations of the face crossing the edge, doubled by the reflection
along it.

By Claim \ref{claim:contributionstodirichlet}, this inequality implies
\begin{align*}
\mu\left(c_{yz}\left(\grad_{yz}f\right)^{2}\right) & \le cm^{d-1}\left(e^{-\lambda m}\,m^{2}/L^{2}+m^{d-1}\,e^{-\lambda m}\,m^{2}/L^{2}\right)\le e^{-\lambda m}/L^{2}.
\end{align*}
In conclusion,
\[
\cD(f)=\mu\left[\sum_{x\in\partial\Lambda}\text{Var}_{x}(f)+\sum_{z\sim y}c_{yz}\left(\grad_{yz}f\right)^{2}\right]\le e^{-\lambda m}L^{d-2}.
\]
\end{proof}
Propositions \ref{prop:bigvar} and \ref{prop:smalldirichlet} show
that indeed that the relaxation time is greater than $e^{\lambda m}L^{2}$,
which by the choice of $m$ coincides with the lower bound of Theorem \ref{maintheorem:1}. \qed
\section{Concluding remarks and further questions}
The general scheme of the proof has already been proven effective in
the study of kinetically constrained spin models (and specifically in
obtaining universality results \cite{MMT}).
It consists in analysing the microscopic dynamics up to some mesoscopic scale $\ell$; and then understanding the long range dynamics, which depends on connectivity properties of a percolation process on the lattice of mesoscopic boxes. The long range dynamics depends very weakly on the details of the model. For example, we were able to restrict this dynamics to paths in two dimensions rather than $d$, since already in two dimensions percolation with large enough parameter is supercritical, and satisfies strong enough connectivity properties. 
We believe that applying these methods to other cooperative kinetically constrained lattice gases could yield new interesting results.

In the context of the Kob-Andersen model, the techniques presented in
this paper can be used in order to find the diffusion coefficient of a
marked particle (\cite{ErtulShapira}). We believe that they may also
help understanding further properties of the this model, e.g.,
improving the bound of \cite{CMRT} on the loss of correlation for
local functions, or understanding its hydrodynamic limit. See also
\cite{A.Shapira}*{Chapter 6}.

\vskip 1cm

\end{document}